\documentclass[12pt]{amsart}

\usepackage{latexsym}
\usepackage{amscd}

\usepackage[margin=1.0in]{geometry}

\usepackage{setspace}
\usepackage[parfill]{parskip}
\usepackage{amsmath, amssymb, amsfonts, hyperref, amsthm, bbm, cancel}
\usepackage{wrapfig}
\usepackage{enumitem}
\usepackage{graphicx} 
\newtheorem{thm}{Theorem}[section]
\newtheorem{cor}[thm]{Corollary}
\newtheorem{lem}[thm]{Lemma}
\newtheorem{prop}[thm]{Proposition}

\newenvironment{pf*}[1]{\proof[#1]}{\endproof}
\usepackage{euscript}

\usepackage[OT2,OT1]{fontenc}

\theoremstyle{definition}

\theoremstyle{remark}

\newcommand{\RR}{{\mathbb R}}

\newcommand{\ZZ}{{\mathbb Z}}
\newcommand{\NN}{{\mathbb N}}
\newcommand{\DD}{{\mathbb D}}

\newcommand{\QQ}{{\mathbb Q}}

\newcommand{\cal}[1]{{\mathcal #1}}

\newcommand{\cW}{{\cal W}}

\newcommand{\cF}{{\cal F}}

\newcommand{\cB}{{\cal B}}

\numberwithin{equation}{section}

\begin{document}


\title{Computability of Brjuno-like functions}
\date{January 7, 2025}
\author{Ivan O. Shevchenko and Michael Yampolsky}
\begin{abstract}
In his seminal paper from 1936, Alan Turing introduced the concept of non-computable real numbers and presented examples based on the algorithmically unsolvable Halting problem. We describe a different, analytically natural mechanism for the appearance of non-computability. Namely, we show that additive sampling of orbits of certain skew products over expanding dynamics produces Turing non-computable reals. We apply this framework to Brjuno-type functions to demonstrate that they realize bijections between computable and lower-computable numbers, generalizing previous results of M.~Braverman and the second author for the Yoccoz-Brjuno function to a wide class of examples, including Wilton's functions and generalized Brjuno functions. 

\end{abstract}
\thanks{I.S. was partially supported by NSERC USRA fellowship. M.Y. was partially supported by NSERC Discovery grant}
\maketitle

\section{Introduction}

In 1936, Alan Turing published a seminal paper \cite{Turing} which is rightly considered foundational for modern computer science. The main subject of his paper, as seen from the title, was a concept of a computable real number. Informally, such numbers can be approximated to an arbitrary desired precision using some algorithm. Turing formalized the latter as a Turing Machine (TM) which is now the commonly accepted theoretical model of computation. Appearing before actual computers,  TMs can be somewhat cumbersome to describe, but their computational power is equivalent to those of  programs in a modern programming language, such as {\sl Python}, for instance. Let  $$\DD=\{p2^q, \;p,q\in\ZZ\}$$ denote the set of dyadic rationals.  A number $x\in\RR$ is {\it computable} if there exists a TM $M$, which has a single input $n\in\NN$ and which outputs $d_n\in\DD$ such that
$$|x-d_n|<2^{-n}.$$
The dyadic notation here is purely to support the intuition that modern computers operate in binary; replacing $2^{-n}$ with any other constructive bound, such as, for instance, $n^{-1}$ would result in an equivalent definition.

Since there are only countably many programs in {\sl Python}, there are only countably many computable numbers. Yet, it is surprisingly non-trivial to present an example of one. To this end, Turing introduced an algorithmically unsolvable problem, now known as the Halting Problem: {\sl determine algorithmically whether a given TM halts or runs forever}. Turing showed that there does not exist a program $M$ whose input is another program $M_1$ and whose output is $1$ if $M_1$ halts, and $0$ if $M_1$ does not halt.

From this, a non-computable real is constructed as follows. Let us enumerate all programs $$M_1,M_2,...,M_n,\ldots$$ in some explicit way. For instance, all possible finite combinations of symbols in the {\sl Python} alphabet can be listed in the lexicographic order. Some of these programs will not run at all (and thus halt by definition) but some will run without ever halting. Let the {\it halting predicate} $p(n)$ be equal to $1$ if $M_n$ halts and $0$ otherwise.

Now, set $$\alpha=\sum_{n\geq 1}p(n)3^{-n}.$$
This number is clearly non-computable -- an algorithm computing it could be used to determine the value of the halting predicate for every given $n$, and thus cannot exist.

It is worth noting a further property of $\alpha$. Let us say that $x\in\RR$ is {\it left-computable } if there exists a TM $M$ which outputs an increasing sequence $$a_n\nearrow x.$$ There are, again, only countably many left-computable numbers, and it is trivial to see that computable numbers form their subset (a proper subset, as seen below). 
Right computability is defined in the same way with decreasing sequences, and it is a nice exercise to show that being simultaneously left- and right-computable is equivalent to being computable.

The non-computable number $\alpha$ is left-computable, as it is the limit of
$$a_n=\underset{j\leq n \;|\; M_j\text{ halts in at most }n\text{ steps }}{\sum}3^{-j};$$
which can be generated by a program which emulates  $M_1,\ldots,M_k$ for $k$ steps to produce $a_k$.

Since Turing, examples of non-computable reals were typically constructed along similar lines. In 2000's, working on problems of computability in dynamics, M.~Braverman and the second author discovered that non-computability may be produced via an analytic expression \cite{BY09}. This expression came in the form of the {\it Brjuno function} 
$\cB(\theta)$ which was introduced by J.-C.~Yoccoz \cite{Yoccoz1996} to study linearization problems of irrationally indifferent dynamics. We will discuss $\cB(\theta)$ and its various cousins in detail below, but let us give an explicit formula here. For $\theta\in(0,1)\setminus\QQ$, we have
$$\cB(\theta)=\sum_{n\geq 0}\theta_{-1}\theta_0\theta_1\cdots\theta_{n-1}\log\frac{1}{\theta_n}.$$
Here, $\theta_{-1}=1$, $\theta_0\equiv\theta$, and $\theta_{i+1}$ is obtained from $\theta_i$ by applying the Gauss map $x\mapsto\{1/x\}$. Of course, for rational values of $\theta$ the summand will eventually turn infinite. The sum may also diverge for an irrational $\theta$, and yet can be shown to converge to a finite value for almost all values in $(0,1)$.

It is evident that if $\theta$ is computable then $\cB(\theta)$ is left-computable. Very surprisingly, this can be reversed in the following way. Setting $$y_*\equiv\inf\{\cB(\theta),\;\theta\in(0,1)\},$$ for each left-computable
$y\in[y_*,\infty),$
  there exists a computable $\theta\in(0,1)$ with $\cB(\theta)=y$. Moreover, there is an algorithm which, given a sequence $a_n\nearrow y$, computes such $\theta$. 
  The Brjuno function can be seen as a ``machine'' mapping computable values in $(0,1)$ surjectively onto left-computable values in $[y_*,\infty)$. As we have seen with the example of Turing's $\alpha$, left-computable numbers may be non-computable. Thus, $\cB(\theta)$ describes a natural analytic mechanism for producing non-computable reals.

    The purpose of this paper is as follows. We distill the proof of the above result from \cite{BY09}, where it is somewhat hidden in the considerations of complex dynamics and Julia sets. Moreover, we generalize the result to cover other Brjuno-type functions which have previously appeared in the mathematical literature, some of them 60 years before Yoccoz's work and in a completely different context. Our generalization describes the phenomenon in the language of dynamical systems. As we will see, ergodic sampling of a particular type of dynamics with suitable weights leads to non-computability.

    Finally, for an even broader natural class of functions, which have also previously been studied, and which do not quite fit the above framework, we prove a more general, albeit weaker, non-computability result.
		
		\subsection*{Acknowledgement} The authors would like to thank Stefano Marmi for sharing his insights into Brjuno-like functions.

    \section{Preliminaries}\label{section:Preliminaries}

    \subsection{Computable functions}
    The ``modern'' definition of a computable function requires the concept of an {\it oracle}. Loosely speaking, an oracle for a real number $x$, for example, is a user who {\it knows $x$} and, when queried by a TM, can input its value with any desired precision. In the world of Turing Machines (as well as {\sl Python} programs) an oracle may be conceived as an infinite tape on which an infinite string of dyadic rationals is written (encoding, for instance, a Cauchy sequence for $x\in\RR$) and which the program is able to read at will. Of course, only a finite amount of information can be read off this tape each time. As well as a real number, an oracle can be used to encode anything else which could be written on an infinite tape, for instance, the magically obtained solution to the Halting Problem.

    Formally, an oracle is  a function $\phi:\NN\to\DD$. An {\it oracle for $x\in\RR$} satisfies
\begin{equation*}
    |\phi(n) - x| < 2^{-n} \text{\hspace{10px} for all \hspace{10px}} n \in \mathbb{N}.
\end{equation*}
We say a TM $M^{\phi}$ is an \textit{oracle Turing Machine} if at any step of the computation, $M^{\phi}$ can query the value $\phi(n)$ for any $n$. We treat an oracle TM as a function of the oracle; that is, we think of $\phi$ in $M^{\phi}$ as a placeholder for any oracle, and the TM performs its computational steps depending on the particular oracle it is given. We will talk about ``querying the oracle'', ``being given the access to an oracle for $x$'', or just ``given $x$''.

We need oracle TMs to define computable functions on the reals. For $S\subset \RR$, we say that {\it a function $f:S\to\RR$ is computable} if there exists an oracle TM $M^\phi$ with a single input $n$ such that for any $x\in S$ the following is true. If $\phi$ is an oracle for $x$, then upon input $n$, the machine $M^\phi$ outputs $d_n\in\DD$ such that
$$|f(x)-d_n|<2^{-n}.$$
In other words, there is an algorithm which can output the value of $f(x)$ with any desired precision if it is allowed to query the value of $x$ with an arbitrary finite precision.

The domain of the real-valued function plays an important role in the above definition. The definition states that there is a single algorithm which, given $x$, works for every $x \in S$. We will abbreviate this by saying that $f$ is {\it uniformly computable on $S$}.
The weakest computability result and the strongest non-computability result one can obtain in regards to real-valued functions is then when the domain is restricted to a single point.

It is worth making note of the following easy fact, whose proof we leave as an exercise:
\begin{prop}
  \label{prop:computability_continuity}
If $f$ is uniformly computable on $S$ then $f$ is continuous on $S$.
  \end{prop}

Uniform left- or right- computability of functions is defined in a completely analogous way.

\subsection{Brjuno function and friends}
Every irrational number $\theta$ in the unit interval admits a unique (simple) continued fraction expansion:
\begin{equation*}
    [a_{1}, a_{2}, a_{3}, \ldots] \equiv \cfrac{1}{a_{1}+\cfrac{1}{a_{2}+\cfrac{1}{a_{3}+\ldots}}} \in (0,1)\setminus\mathbb{Q},
\end{equation*}
where $a_{i} \in \NN$. An important related concept is that of the \textit{Gauss map} $G:(0,1]\to [0,1]$ given by 
\begin{equation*}
    G(x) = \left\{\frac{1}{x}\right\};
\end{equation*}
it has the property $$G([a_{1}, a_{2}, a_{3}, \ldots]) = [a_{2}, a_{3}, a_{4}, \ldots].$$

In what follows, for a function $F$, we denote $F^{n}$ its $n$-th iterate.
For ease of notation, for each $j\in\NN$ let us define the function
$\eta_{j}: (0, 1) \setminus\mathbb{Q} \to (0, 1) \setminus\mathbb{Q}$ as 
$ \eta_{j}(x) = G^{j-1}(x)$,
so that $$\eta_{j}([a_{1}, a_{2}, a_{3}, \ldots]) = [a_{j}, a_{j+1}, a_{j+2}, \ldots].$$ We define \textit{Yoccoz's Brjuno function}, or for brevity just the \textit{Brjuno function} \cite{Yoccoz1996}  by
\begin{equation} \label{eq:YoccozBrjuno}
    \cB(x) = \sum_{i=1}^{\infty} \eta_{0}(x) \eta_{1}(x) \cdots \eta_{i-1}(x) \cdot (-\log(\eta_{i}(x))),
\end{equation}

where we set $\eta_{0}(x) = 1$ for all $x$. Irrationals in $(0,1)$ for which $\cB(x)<+\infty$ are known as {\it Brjuno numbers}; they form a full measure subest of $(0,1)$. 
The original work of Brjuno \cite{Brjuno1971} characterized Brjuno numbers using a different infinite series, whose convergence is equivalent to that of (\ref{eq:YoccozBrjuno}).

The Brjuno condition has been introduced in the study of linearization of neutral fixed points. The function $\cB$ has an important geometric meaning in this context as an estimate on the size of the domain of definition of a linearizing coordinate. It
has a number of remarkable properties, and has been studied extensively, see for instance \cite{MMY}.

Intuitively, the condition $\cB(x)<+\infty$ is a Diophantine-type condition; if $x$ is a Diophantine number then it can be shown that the series (\ref{eq:YoccozBrjuno}) is majorized by a geometric series.
As we have learned from a talk by S.~Marmi \cite{talkMarmi}, similar expressions have appeared much earlier in the theory of Diophantine approximation. Notably, in 1933 Wilton \cite{Wilton1933} defined the sums
\begin{align}
    \cW_{1}(x) &= \sum_{i=1}^{\infty} \eta_{0}(x) \eta_{1}(x) \cdots \eta_{i-1}(x) \cdot (-\log^{2}(\eta_{i}(x))), \label{eq:Wilton1} \\
    \cW_{2}(x) &= \sum_{i=1}^{\infty} (-1)^{i+1} \cdot \eta_{0}(x) \eta_{1}(x) \cdots \eta_{i-1}(x) \cdot (-\log(\eta_{i}(x))) \label{eq:Wilton2}
\end{align}
which we will call {\it the first and second Wilton functions} respectively.

To illustrate, how different the application of Wilton's functions is from the Brjuno function, let us quote Wilton's results. If we denote $d(n)$ to be the number of divisors of a positive integer $n$, Wilton showed that
\begin{align*}
    \sum_{n=1}^{\infty} \dfrac{d(n)}{n} \cos 2\pi nx < \infty &\text{\hspace{10px} if and only if \hspace{10px}} \cW_{1}(x) < \infty, \text{\hspace{5px} and} \\
    \sum_{n=1}^{\infty} \dfrac{d(n)}{n} \sin 2\pi nx < \infty &\text{\hspace{10px} if and only if \hspace{10px}} \cW_{2}(x) < \infty.
\end{align*}

One important generalization of the simple continued fraction expansion of an irrational number is the \textit{$\alpha$-continued fraction expansion} for $\alpha \in [1/2, 1]$. Let $A_{\alpha} : [0, \alpha] \to [0, \alpha]$ be the map
\begin{equation*}
    A_{\alpha}(0) = 0, \text{\hspace{10px}} A_{\alpha}(s) = \left| \dfrac{1}{x} - \left\lfloor \dfrac{1}{x} - \alpha + 1 \right\rfloor  \right|, x \neq 0.
\end{equation*}
By iterating this mapping, we define the infinite $\alpha$-continued fraction expansion for any $x \in (0, \alpha) - \mathbb{Q}$ as follows. For $n \geq 0$ we let
\begin{align*}
    & x_{0} = |x - \lfloor x - \alpha + 1 \rfloor|, \text{\hspace{31px}} a_{0} = \lfloor x - \alpha + 1  \rfloor, \text{\hspace{60px}} \varepsilon_{0} = 1,\\
    & x_{n+1} =  A_{\alpha}(x_{n}) = A_{\alpha}^{n+1}(x), \text{\hspace{15px}} a_{n+1} = \left\lfloor \dfrac{1}{x_{n}} - \alpha + 1 \right\rfloor \geq 1, \text{\hspace{15px}} \varepsilon_{n+1} = \text{sgn}(x_{n}).
\end{align*}
Then we can write
\begin{equation*}
    x = [(a_{1}, \varepsilon_{1}), (a_{2}, \varepsilon_{2}), \ldots, (a_{n}, \varepsilon_{n}), \ldots] := \dfrac{\varepsilon_{0}}{a_{1}+\dfrac{\varepsilon_{1}}{a_{2}+\dfrac{\varepsilon_{2}}{a_{3}+\ldots}}} \in (0,\alpha)-\mathbb{Q}.
\end{equation*}
Note that when $\alpha=1$, we recover the standard continued fraction expansion.

Generalizations of the Brjuno function based on the above expansions have been studied, for example, in \cite{Luzzi2010}. There the authors considered the properties of the function
\begin{equation} \label{eq:B_au}
    \cB_{\alpha, u}(x) = \sum_{i=1}^{\infty} \eta_{\alpha,0}(x) \eta_{\alpha,1}(x) \cdots \eta_{\alpha,i-1}(x) \cdot u(\eta_{\alpha,i}(x)),
\end{equation}
where $\eta_{\alpha,0}(x) = 1$, $\eta_{\alpha, j}(x) = A_{\alpha}^{j-1}(x), \alpha \in [1/2, 1], j \geq 1$ is a generalization of $\eta_{j}$ to alpha-continued fraction maps, and $u:(0, 1) \to \mathbb{R}^{+}$ is a $\mathcal{C}^{1}$ function such that
\begin{equation*}
    \lim_{x\to0^{+}}u(x) = \infty, \text{\hspace{10px}} \lim_{x\to0^{+}}x \cdot u(x) < \infty, \text{\hspace{10px}} \lim_{x\to0^{+}}x^{2} \cdot u'(x) < \infty.
\end{equation*}
A further generalized class $\{\cB_{\alpha, u, \nu}\}$ of Brjuno functions is discussed in \cite{Bakhtawar2024}
where the last two conditions above are dropped and the term $\eta_{\alpha,1}(x) \eta_{\alpha,2}(x) \cdots \eta_{\alpha,i-1}(x)$ is raised to some power $\nu \in \mathbb{Z}^{+}$:
\begin{equation} \label{eq:B_auv}
    \cB_{\alpha, u, \nu}(x) = \sum_{i=1}^{\infty} (\eta_{\alpha,0}(x) \eta_{\alpha,1}(x) \cdots \eta_{\alpha,i-1}(x))^{\nu} \cdot u(\eta_{\alpha,i}(x)).
\end{equation}
As shown in \cite{Bakhtawar2024}:
\begin{prop}\label{prop:Brjuno-max}
  For all $\alpha \in \mathbb{Q} \cap [1/2, 1]$, such functions $\cB_{\alpha, u, \nu}$ are lower semi-continuous and thus attain their global minima.
\end{prop}
  Note that if we take $\alpha=1$, $\nu=1$, and $u = -\log$, we recover Yoccoz's Brjuno function discussed above; similarly, taking $\alpha=1$, $\nu=1$, and $u = -\log^{2}$ recovers the first Wilton function $\cW_{1}$. Note, however, that we cannot obtain $\cW_{2}$ as a special case of $\cB_{\alpha, u, \nu}$.

  \section{Statements of the results}\label{section:Statements}
  \subsection{A general framework}
  We refer the readers to the survey \cite{surveyBrjuno} which discusses a cohomological interpretation of the Brjuno function and lays the groundwork for its generalizations.
Our discussion will be much less technical, yet essentially equivalent in the cases we consider. It will yield a generalization which is 
(a) broad enough to include the relevant examples we have quoted and (b) captures the essence of the non-computability phenomenon discovered in  \cite{BY09}.  This is achieved via the following framework. Suppose $G(x)$ is a piecewise-defined expanding mapping whose domain is an infinite collection of subintervals of $(0,1)$ each of which is mapped surjectively over all of $(0,1)$. Let $\psi:(0,1)\to\RR$ (in our case, $\psi(x)=x^\nu$), and consider the skew product dynamics
given by
\begin{equation}\label{eq:skewprod}
F:\left(\begin{array}{c}x\\y\end{array}\right)=\left(\begin{array}{c}G(x)\\\psi(x) y\end{array}\right)
  \end{equation}
The class of functions $\cF$ for which non-computability arises is produced by additive (ergodic) sampling of orbits $(x_n,y_n)$ of $F$ using a suitable weight function
$u(x,y)\equiv u(x)$ with positive values:
\begin{equation}
  \label{eq:genform}
    \cF(x)=\sum_{i=1}^{\infty} y_n \cdot u(x_n).
\end{equation}

It is worth noting that such a function is a formal solution of the twisted cohomological equation
$$\cF-\psi \cdot \cF\circ G=u.$$
As noted above, in our case we set
$$\psi(x)\equiv x^\nu.$$
We will now need to make somewhat technical but straightforward general assumptions on $G$ and $u$.



Let $G: (0, 1) \to (0, 1)$ be a function which is $\mathcal{C}^{1}$ on a set $S \subseteq (0, 1)$, with $s_{0} := \inf S$, $s_{1} := \sup S$. Suppose that for a countable collection of disjoint open intervals $J_{i} = (\ell_{i}, r_{i})$ with $\ell_{1} > \ell_{2} > \cdots$, we have $S = J_{1} \sqcup J_{2} \sqcup J_{3} \sqcup \cdots$. Below we will denote $G_{i} := G|_{J_{i}}$, so that $G_{i}^{-1}: (s_{0}, s_{1}) \to J_{i}$ is the unique branch of $G^{-1}$ mapping into $J_{i}$. We assume $G$ satisfies the following criteria.

\begin{enumerate} [label={(\roman*)}]
    \item \label{G:I_i surjects} $G(J_{i}) = (s_{0}, s_{1})$ for each $i$.
    \item \label{G:expansivity} $|G'|>1$, and additionally there exist $\tau>1, \sigma >1$, and $\kappa \in \mathbb{Z}$ such that if we let
    \begin{equation*}
        \tau_{i, 1} := \inf_{x \in J_{i}} \left| G'(x) \right|, \text{\hspace{10px}} \tau_{i, \kappa} := \inf_{x \in J_{i}} \left| (G^{\kappa})'(x) \right|,
    \end{equation*}
    then we have both $\tau_{i, 1}^{-1} < \ell_{i} \cdot \sigma$ and $\tau_{i, \kappa}^{-1} < \ell_{i} \cdot \tau^{-1}$ for all $i$. In particular, since $\ell_{i} \leq s_{1} < 1$, this means that $|(G^{\kappa})'(x)| > \tau$ for all $x \in (s_{0}, s_{1})$.
    \item \label{G:decreasing} $G_{1}$ is decreasing on $J_{1}$.
    \item \label{G:l_i and r_i condition} Let $\varphi$ be the unique fixed point of $G_{1}$, and let $\delta_{G}(N) := G_{N}^{-1}(\varphi)-G_{N+1}^{-1}(\varphi)>0$. Then there is some constant $D>0$ such that for all $N \in \mathbb{Z}^{+}$,
    \begin{equation*}
        \dfrac{r_{N+1}}{\ell_{N+1}} \cdot \dfrac{r_{N}-\ell_{N+1}}{\delta_{G}(N)} < D.
    \end{equation*}
    \item \label{G:l_i and r_i condition squared} We have $\dfrac{r_{i}-\ell_{i}}{\ell_{i}^{2}} < D$ for some constant $D>0$ independent of $i$.
    \item \label{G:computability} $G$ is computable on its domain.
    \item \label{G:function g} $g(i) := \dfrac{r_{i}}{\ell_{i+1}} - 1 \to 0$ as $i \to \infty$. Note that since $g$ is positive and bounded from above, there is some constant $m_{g}>0$ for which $0 < g(i) \leq m_{g}$.
\end{enumerate}
It is easy to show that properties \ref{G:I_i surjects} and \ref{G:expansivity} imply that $G$ restricted to the set $\Lambda = \bigcap_{j=0}^{\infty}G^{-j}((s_{0}, s_{1}))$ is topologically conjugate to the full shift over the alphabet of positive integers $\mathbb{Z}^{+}$. As such we can write any $x \in \Lambda$ as its symbolic representation $x = [a_{1}, a_{2}, \ldots]$, where $G^{j-1}(x) \in J_{a_{j}}$. In this notation, $\varphi$ above can be written as $[1, 1, 1, \ldots]$. By \ref{G:decreasing} we have $\varphi < s_{1}$, from which it is straightforward to check that we indeed have $\delta_{G}(N)>0$ in \ref{G:l_i and r_i condition}. For convenience, for $j \in \mathbb{Z}^{+}$ we will denote $\eta_{j}: \Lambda \to \Lambda$, $\eta_{j} = G^{j-1}$, so that $$\eta_{j}([a_{1}, a_{2}, \ldots]) = [a_{j}, a_{j+1}, \ldots].$$ Henceforth we will assume $G$ is restricted to $\Lambda$.

Now, let $u: (s_{0}, s_{1}) \to \mathbb{R}^{+}$ be a $\mathcal{C}^{1}$ function satisfying the following:
\begin{enumerate}[label={(\roman*)}]
    \item \label{u:1} $\lim_{x \to s_{0}^{+}}u(x) = \infty$.
    \item \label{u:2} If $s_{1}=1$, then $\liminf_{N \to \infty} \inf_{z,w \in (s_{0},s_{1})} \dfrac{u \circ G_{1}^{-1} \circ G_{N}^{-1}(z)}{u \circ G_{1}^{-1} \circ G_{N}^{-1}(w)} > 0$.
    \item \label{u:3} There is some $C>0$ such that $|u'(x)| < \dfrac{C}{(x-s_{0})^{2}}$ for all $x \in (s_{0}, s_{1})$.
    \item \label{u:4} $u$ is computable on numbers $[a_{1}, a_{2}, \ldots]$ such that $a_{l}=1$ for all $l>l_{0}$ for some integer $l_{0}$.
    \item \label{u:5} $u$ is left-computable on all of $\Lambda$.
\end{enumerate}

In what follows, for any $\nu >0$ we define the \textit{generalized Brjuno function} to be
\begin{equation*}
    \Phi(x) := \sum_{i=1}^{\infty} \left( \eta_{0}(x) \cdots \eta_{i-1}(x) \right)^{\nu} \cdot u(\eta_{i}(x)),
\end{equation*}
where  $x = [a_{1}, a_{2}, \ldots] \in \Lambda$.


We note:

\begin{thm}\label{thm:included}
The following functions restricted to their corresponding sets $\Lambda$ fall under the definition of a generalized Brjuno function given above:
\begin{itemize}
    \item Yoccoz's Brjuno function $\cB$ (\ref{eq:YoccozBrjuno}).
    \item The first Wilton function $\cW_{1}$ (\ref{eq:Wilton1}).
    \item The functions $\cB_{\alpha, u, \nu}$ (\ref{eq:B_auv}) under the additional assumptions of \ref{u:2}-\ref{u:5} on $u$. For example, taking $G = A_{\alpha}$ for $\alpha \in [1/2,1]$ and $u(x)$ to be any of $\log^{n}(1/x)$ for $n \in \mathbb{Z}$ or $x^{-1}$ yields a generalized Brjuno function.
\end{itemize}
\end{thm}

Let us postpone the proof to \S~\ref{sec-proof-included} in the Appendix, and proceed to formulating the results.

\subsection{Main results}

\begin{thm}\label{thm:main}
  Let $x_*$ be a computable real number in $\Lambda$ with the property $y_*=\Phi(x_*)<+\infty.$ 
  There exists an oracle TM $M^\phi$ with a single input $n\in\NN$ such that the following holds. Suppose $y\in[y_*,+\infty)$ and
    $$\phi(n)=y_{n}\nearrow y.$$
    Then $M^\phi$ outputs $d_n$ such that
    $$|d_n-x|<2^{-n}\text{ for }x \in \Lambda\text{ such that }\Phi(x)=y.$$
\end{thm}

As a corollary:
\begin{cor}
  If $y\in[y_*,+\infty)$ is left computable then there exists a computable $x\in \Lambda$ with
    $$\Phi(x)=y.$$
\end{cor}
In fact, it is clear from the proof of Theorem~\ref{thm:main} that countably many such $x \in \Lambda$ exist.

As was shown in \cite{BM2020}, the Brjuno function $\cB$ attains its global maximum at
$$w_*=\frac{\sqrt{5}-1}{2}=[1,1,1,\ldots].$$
Therefore:
\begin{cor}
  If $y\in[\cB(w_*),+\infty)$  is left computable then there exists a computable $x\in(0,1)$ with
    $$\cB(x)=y.$$

\end{cor}

\section{Proof of Theorem~\ref{thm:main}}\label{section:Proof of Main Theorem}

\subsection{Three main lemmas}

It will be helpful to first outline the general strategy for the proof of Theorem~\ref{thm:main}. We are given a computable sequence $\{y_{n}\}$ converging upwards to some left-computable $y$, and we must find a computable $x \in (s_{0}, s_{1})$ for which $y = \Phi(x)$. This is done by starting with some $\gamma_{0} \in (s_{0}, s_{1})$ and iteratively modifying the symbolic representation of $\gamma_{k}$ to "squeeze" $\Phi(\gamma_{k})$ to be in the interval $(\Phi(\gamma_{s+k})-2^{-k}\varepsilon, \Phi(\gamma_{s+k})+2^{-k}\varepsilon)$ for some positive integer $s$. Passing to the limit, we obtain $x = \gamma_{\infty}$ for which $\Phi(x) = y$ as needed.

To ensure this strategy works, we need ways of carefully controlling the value of $\Phi(\gamma)$ from the symbolic representation of $\gamma$. This role is played by Lemmas \ref{lem:5.18}, \ref{lem:5.19}, and \ref{lem:5.20} below, which are analogous to Lemmas 5.18, 5.19, and 5.20 respectively in \cite{BY09} and whose proofs are in the subsection~\ref{sec-proof-main-lemmas} in the Appendix.

\begin{lem}\label{lem:5.18}
For any initial segment $I = [a_{1}, a_{2}, \ldots, a_{n}]$, write $\omega = [a_{1}, a_{2}, \ldots, a_{n}, 1, 1, 1, \ldots]$. Then for any $\varepsilon > 0$, there is an $m > 0$ and an integer $N$ such that if we write $\beta^{N} = [a_{1}, a_{2}, \ldots, a_{n}, 1, 1, \ldots, 1, N, 1, 1, \ldots]$, where the $N$ is located in the $(n+m)$-th position, then
\begin{equation*}
    \Phi(\omega) + \varepsilon < \Phi(\beta^{N}) < \Phi(\omega) + 2 \varepsilon.
\end{equation*}
\end{lem}

In the appendix, this Lemma will be proven under the slightly weaker assumptions used in section \ref{section:Generalized non-computability result}. Also, it is clear from the proof that $m$ can be taken arbitrarily large.

\begin{lem}\label{lem:5.19}
With $\omega$ as above, for any $\varepsilon>0$ there is an $m_{0}>0$, which can be computed from $(a_{0}, a_{1}, \ldots, a_{n})$ and $\varepsilon$, such that for any $m \geq m_{0}$ and any tail $I = [a_{n+m}, a_{n+m+1}, \ldots]$ we have
\begin{equation*}
    \Phi(\beta^{I}) > \Phi(\omega) - \varepsilon
\end{equation*}
\textit{where}
\begin{equation*}
    \beta^{I} = [a_{1}, a_{2}, \ldots, a_{n}, 1, 1, \ldots, 1, a_{n+m}, a_{n+m+1}, \ldots].
\end{equation*}
\end{lem}

\begin{lem}\label{lem:5.20}
Let $\omega = [a_{1}, a_{2}, \ldots]$ be such that $\Phi(\omega) < \infty$. Write $\omega_{k} = [a_{1}, a_{2}, \ldots, a_{k}, 1, 1, \ldots]$. Then for every $\varepsilon>0$ there is an $m$ such that for all $k \geq m$,
\begin{equation*}
    \Phi(\omega_{k}) < \Phi(\omega) + \varepsilon.
\end{equation*}
\end{lem}

We proceed with proving Theorem~\ref{thm:main}.

\subsection{The proof}

We first require a few preliminary results.

\begin{prop}\label{prop:computes_Brjuno_value}
Let $m_{0}>0$ be an integer. Then there exists an oracle TM which, given access to a number $x = [a_{1}, a_{2}, \ldots] \in \Lambda$ such that $a_{m} = 1$ for $m>m_{0}$, computes $\Phi(x)$.
\end{prop}

\begin{proof}
Let $\varphi = [1, 1, 1, \ldots]$. For any $x$ we have
\begin{align*}
    \Phi(x) &= \sum_{i=1}^{m_{0}} \left( \eta_{0}(x) \cdots \eta_{i-1}(x) \right)^{\nu} \cdot u(\eta_{i}(x)) + \sum_{i=m_{0}+1}^{\infty} \left( \varphi^{\nu} \right)^{n-1} \cdot u(\varphi) \\
&= \sum_{i=1}^{m_{0}} \left( \eta_{0}(x) \cdots \eta_{i-1}(x) \right)^{\nu} \cdot u(\eta_{i}(x)) + \frac{\varphi^{\nu \cdot m_{0}}}{1-\varphi^{\nu}} \cdot u(\varphi).
\end{align*}
Since $u$ is computable on each number $\eta_{i}(x)$ whose symbolic representation ends in all ones by assumption \ref{u:4} on $u$, and each $\eta_{j}$ is computable on all of $\Lambda$ by assumption \ref{G:computability} on $G$, there is an oracle TM which, given access to $x$, computes the sum on the left to an arbitrary precision. The construction of this oracle TM is independent of $x$, and depends only on $m_{0}$. Additionally, since $\varphi$ also ends in all ones there is a TM which computes the value of the term on the right to an arbitrary precision, also depending only on $m_{0}$. Combining these Turing Machines in the obvious way we obtain an oracle TM which, given access to $x$, computes $\Phi(x)$ to an arbitrary precision.
\end{proof}

\begin{lem}\label{lem:uniformly_compute_approx}
Given an initial segment $I = [a_{0}, a_{1}, \ldots, a_{n}]$ and $m_{0}>0$, write \\ $\omega = [a_{0}, a_{1}, \ldots, a_{n}, 1, 1, \ldots]$. Then for all $\varepsilon>0$, we can uniformly compute $m>m_{0}$ and $t, N \in \mathbb{Z}^{+}$ such that if we write $\beta = [a_{0}, a_{1}, \ldots, a_{n}, 1, 1, \ldots, 1, N , 1, 1, \ldots]$, where $N$ is in the $(n+m)$-th position, we have:
\begin{equation} \label{eq:1'}
    \Phi(\omega) + \varepsilon < \Phi(\beta) < \Phi(\omega) + 2\varepsilon,
\end{equation}
and for any $\gamma = [a_{0}, a_{1}, \ldots, a_{n}, 1, 1, \ldots, 1, N , 1, \ldots, 1, c_{n+m+t+1}, c_{n+m+t+2}, \ldots]$, we have
\begin{equation} \label{eq:2'}
    \Phi(\gamma) > \Phi(\omega) - 2^{-n}.
\end{equation}
\end{lem}

\begin{proof}
We will first show that such $m, N$ exist, then give an algorithm to compute them. Let $\varepsilon>0$ be given. By Lemma~\ref{lem:5.18}, there exists $m$ (and we can make $m>m_{0}$) and $\exists N \in \mathbb{Z}^{+}$ such that
\begin{equation*}
    \Phi(\omega) + \varepsilon < \Phi(\beta) < \Phi(\omega) + 2\varepsilon.
\end{equation*}
Taking $I'=[a_{0}, a_{1}, \ldots, a_{n}, 1, 1, \ldots, 1, N]$ and $\omega'=\beta$, applying Lemma~\ref{lem:5.19} with $\varepsilon'=2^{-n}$ we get $t_{0}>0$ which can be computed from $(a_{0}, a_{1}, \ldots, a_{n}, 1, 1, \ldots, 1, N)$ such that for all $t \geq t_{0}$ and any tail $I = [c_{n+m+t+1}, c_{n+m+t+2}, \ldots]$ we have
\begin{equation*}
    \Phi(\gamma) > \Phi(\omega') -2^{-n} = \Phi(\beta)-2^{-n} > \Phi(\omega)-2^{-n}
\end{equation*}
as needed.

Since $\omega, \beta$ have symbolic representations ending in all ones, for any specific $m,N$ we can compute $\Phi(\omega)$ and $\Phi(\beta)$ by Proposition~\ref{prop:computes_Brjuno_value}. So, we can find the required $m,N$ by enumerating all pairs $(m,N)$ and exhaustively checking equations (\ref{eq:1'}), (\ref{eq:2'}) for each of them. We know that we will eventually find a pair for which these equations hold. Once we have $m$ and $N$, we can use Lemma~\ref{lem:5.19} to compute $t$.
\end{proof}

\begin{lem}\label{lem:inf_over_heads}
The infimum $\Phi(x_{*})$ of $\Phi(x)$ over all $x \in \Lambda$ is equal to the infimum over the numbers whose symbolic representations have only finitely many terms that are not $1$:
\begin{equation*}
    \Phi(x_{*}) = \inf_{x=[a_{1},a_{2}, \ldots, a_{k}, 1, 1, \ldots]} \Phi(x).
\end{equation*}
\end{lem}

\begin{proof}
Let $\varepsilon>0$ be given. By definition of infimum, there exists $x=[a_{1}, a_{2}, \ldots]$ such that
\begin{equation*}
    \Phi(x) < \Phi(x_{*}) + \frac{\varepsilon}{2}.
\end{equation*}
Write $x_{k}=[a_{1}, a_{2}, \ldots, a_{k}, 1, 1, \ldots]$. By Lemma~\ref{lem:5.20}, there exists $m$ such that for $k \geq m$,
\begin{equation*}
    \Phi(x_{k}) < \Phi(x) + \frac{\varepsilon}{2}.
\end{equation*}
Thus $\Phi(x_{k})< \Phi(x_{*}) + \varepsilon$, so we can make $\Phi(x_{k})$ as close to $\Phi(x_{*})$ as we need.
\end{proof}

Now, we are given 
\begin{equation*}
    y_{n}\nearrow y, \hspace{15px} y\in[y_*,+\infty).
\end{equation*}

The case of $y = y_{*}$ is trivial, so we suppose $y > y_{*}$. Then there is an $s$ and an $\varepsilon>0$ such that 
\begin{equation*}
    y_{s} > \Phi(x_{*}) + 2\varepsilon.
\end{equation*}
By Lemma~\ref{lem:inf_over_heads}, there exists $\gamma_{0}=[a_{1}, a_{2}, \ldots, a_{n}, 1, 1, \ldots]$ such that
\begin{equation*}
    y_{s} - \varepsilon < \Phi(\gamma_{0}) < y_{s} - \frac{\varepsilon}{2}.
\end{equation*}
We will now give an algorithm for computing a number $x \in \Lambda$ for which $\Phi(x) = \lim \nearrow y_{n} = y$, which would complete the proof of Theorem~\ref{thm:main}. The algorithm works as follows. At state $k$ it produces a finite initial segment $I_{k}=[a_{0}, \ldots, a_{k}]$ such that the following properties hold:

\begin{enumerate}[label={(\arabic*)}]
    \item \label{item_list:1} $I_{0} = [a_{1}, a_{2}, \ldots, a_{n}]$.
    \item \label{item_list:2} $I_{k}$ has at least $k$ terms, i.e. $m_{k} \geq k$.
    \item \label{item_list:3} For each $k$, $I_{k+1}$ is an extension of $I_{k}$.
    \item \label{item_list:4} For each $k$, define $\gamma_{k}=[I_{k}, 1, 1, \ldots]$. Then
        \begin{equation*}
            y_{s+k} - 2^{-k} \varepsilon < \Phi(\gamma_{k}) < y_{s+k} - 2^{-(k+1)} \varepsilon.
        \end{equation*}
    \item \label{item_list:5} For each $k$, $\Phi(\gamma_{k}) > \Phi(\gamma_{k+1})$.
    \item \label{item_list:6} For each $k$ and any extension $\beta = [I_{k}, b_{m_{k}+1}, b_{m_{k}+2}, \ldots]$, we have
        \begin{equation*}
            \Phi(\beta) > \Phi(\gamma_{k})-2^{-k}.
        \end{equation*}
\end{enumerate}

The first three properties are easy to verify. The last three are checked using Lemma~\ref{lem:uniformly_compute_approx}. By this Lemma we can increase $\Phi(\gamma_{k-1})$ by any given amount, possibly in more than one step, by extending $I_{k-1}$ to $I_{k}$. Thus if we have
\begin{equation*}
    y_{s+k-1} -2^{-(k-1)} \varepsilon < \Phi(\gamma_{k-1}) < y_{s+k-1} -2^{-k} \varepsilon,
\end{equation*}
by virtue of $\{ a_{s+k} \}_{k=1}^{\infty}$ being non-decreasing we have both
\begin{equation*}
    y_{s+k-1} - 2^{-(k-1)} \varepsilon < y_{s+k} - 2^{-k} \varepsilon \hspace{15px} \text{and} \hspace{15px} y_{s+k-1} - 2^{-k} \varepsilon < y_{s+k} - 2^{-(k+1)} \varepsilon.
\end{equation*}
So, we can increase $\Phi(\gamma_{k-1})$ by such a fine amount that
\begin{equation*}
    y_{s+k} -2^{-k} \varepsilon < \Phi(\gamma_{k}) < y_{s+k} -2^{-(k+1)} \varepsilon,
\end{equation*}
satisfying the fourth and fifth properties. In performing this fine increase, we have used the fact that the $y_{s+k}$'s are computable. The last property is satisfied by Lemma~\ref{lem:uniformly_compute_approx} (\ref{eq:2'}).

Denote
\begin{equation*}
    x = \lim_{k \to \infty} \gamma_{k}.
\end{equation*}
The symbolic representation of $x$ is the limit of the initial segments $I_{k}$. This algorithm gives us at least one term of the symbolic representation of $x$ per iteration, and hence we would need at most $O(n)$ iterations to compute $x$ with precision $2^{-n}$. The initial segment of $\gamma_{0}$ can also be computed as in the proof of Lemma~\ref{lem:uniformly_compute_approx}. It remains to show that $x$ is the number we are after.

\begin{lem}\label{lem:limit_equals_y}
We have $\Phi(x) = y$.
\end{lem}

\begin{proof}
Taking limits on all sides of \ref{item_list:4}, we get 
\begin{equation*}
    \lim_{k\to\infty} \Phi(\gamma_{k}) = \lim_{k\to\infty} y_{k} = y.
\end{equation*}
It remains to show $\lim_{k\to\infty} \Phi(\gamma_{k}) = \Phi(x)$.
As in Proposition~\ref{prop:computes_Brjuno_value}, denote $x = [a_{1}, a_{2}, \ldots]$ and let $\alpha_{k}=[a_{k}, a_{k+1}, \ldots], \alpha_{0}=1$. Let $\varphi = [1, 1, 1, \ldots]$, and additionally for any number $\xi = [b_{1}, b_{2}, \ldots] \in \Lambda$ denote $(\xi)_{k} = [b_{k}, b_{k+1}, \ldots]$. We have
\begin{align*}
    \Phi(x) &= \lim_{k\to\infty} \left( \sum_{n=1}^{m_{k}} \alpha_{0} \alpha_{1} \cdots \alpha_{n-1} \cdot u(\alpha_{n}) \right) \\
    &\leq \lim_{k\to\infty} \left( \sum_{n=1}^{m_{k}} \alpha_{0} \alpha_{1} \cdots \alpha_{n-1} \cdot u(\alpha_{n}) + \sum_{n=m_{k}+1}^{\infty} \varphi^{n-1} \cdot u(\varphi) \right) \\
    &= \lim_{k\to\infty} \Phi(\gamma_{k}).
\end{align*}
Additionally, taking limits on both sides of \ref{item_list:6} with $\beta = x$ yields
\begin{equation*}
    \lim_{k\to\infty} \Phi(\gamma_{k}) \leq \Phi(x),
\end{equation*}
therefore $\lim_{k\to\infty} \Phi(\gamma_{k}) = \Phi(x)$.
\end{proof}

This concludes the proof of Theorem~\ref{thm:main}.

\section{Generalized non-computability result} \label{section:Generalized non-computability result}

\subsection{Modified assumptions}

We will now prove a non-computability result about a slightly more broad class of generalized Brjuno functions than the one considered in section \ref{section:Generalized non-computability result}. For this result we require all the assumptions on $G$ except \ref{G:l_i and r_i condition squared} and \ref{G:computability}, and on $u$ we require only assumption \ref{u:1} together with a slightly weaker variation of assumption \ref{u:3}:
\begin{enumerate} [label={(\roman*$'$)}]
    \setcounter{enumi}{2}
    \item \label{u:3'} There is some $C>0$ such that $u'(x) < \dfrac{C}{(x-s_{0})^{2}}$ for all $x \in (s_{0}, s_{1})$.
\end{enumerate}
We can additionally allow for more flexibility in the definition of the generalized Brjuno function by adding a "sign" term:
\begin{equation*}
    \Phi(x) := \sum_{i=1}^{\infty} s(i) \cdot \left( \eta_{0}(x) \cdots \eta_{i-1}(x) \right)^{\nu} \cdot u(\eta_{i}(x)),
\end{equation*}
where $\nu >0$ as before and $s(i) \in \{-1, 1\}$.  

\begin{thm}\label{thm:examples_work}
The following functions restricted to their corresponding sets $\Lambda$ fall under the definition of a generalized Brjuno function given above:
\begin{itemize}
    \item Yoccoz's Brjuno function $\cB$ (\ref{eq:YoccozBrjuno}).
    \item The first Wilton function $\cW_{1}$ (\ref{eq:Wilton1}) and the second Wilton function $\cW_{2}$ (\ref{eq:Wilton2}).
    \item The functions $\cB_{\alpha, u, \nu}$ (\ref{eq:B_auv}) under the additional assumption of \ref{u:3'} on $u$. As before, taking $G = A_{\alpha}$ for $\alpha \in [1/2,1]$ and $u(x)$ to be any of $\log^{n}(1/x)$ for $n \in \mathbb{Z}$ or $x^{-1}$ yields a generalized Brjuno function.
\end{itemize}
\end{thm}

The proof is completely analogous to the proof of Theorem~\ref{thm:included} and will be omitted.

\subsection{Noncomputability result}

We proceed with the main result of this section, which concerns computability of the function $\Phi$ as opposed to computability of real numbers.

\begin{thm}\label{thm:non_comp_exists}
There exists a number $x \in \Lambda$ for which $\Phi(\cdot)$ as defined above is not computable at $x$.
\end{thm}

Without loss of generality, we will assume in what follows that the sign term  $s(i) = 1$ infinitely often (otherwise, just replace $\Phi$ with $-\Phi$).

As before, it will be instructive to first go over the strategy of the proof. This outline is rough and is not fully logically sound, however it captures the main idea of the argument. To prove a function is non-computable at a single point $x$, it suffices to enumerate all oracle TMs $M^{\phi}_{i}$, $i \in \mathbb{N}$ (recall that there are countably many oracle TMs), and show that if $\phi$ is any oracle of $x$ then $M^{\phi}_{i}$ does not approximate $\Phi(x)$ arbitrarily well.

We start with $x_{0} = [1, 1, 1, \ldots]$ and the first TM $M^{\phi}_{n_{1}}$ in our enumeration which computes $x_{0}$. If any of the digits $a_{j}$ in the symbolic representation of $x_{0}$ are changed to some $N \in \mathbb{Z}^{+}$, as $N \to \infty$ the series $\Phi(x)$ diverges. However, if we change some digit $a_{j}$ far enough in the representation of $x_{0}$, for any $N$ the new value of $x_{0}$ changes by at most some fixed small amount $\varepsilon_{j}$ which goes to $0$ as $j \to \infty$. So, the idea is define $x_{1}$ from $x_{0}$ by changing $a_{j_{1}}$ for large enough $j_{1}$ to some large enough $N_{1}$, such that if $M^{\phi}_{n_{1}}$ is given an oracle for $x_{1}$ then it does not properly compute $\Phi(x_{1})$, which in some sense "fools" the oracle TM $M^{\phi}_{n_{1}}$. To fool the machine $M^{\phi}_{n_{2}}$ we then change a digit $j_{2}>j_{1}$ sufficiently far in the symbolic representation of $x_{1}$ to a large $N_{2}$ to get $x_{2}$, in such a way that neither $M^{\phi}_{n_{2}}$ nor any other $M^{\phi}_{k}$ for $k < n_{2}$ properly compute $\Phi(x_{2})$. Continuing in this manner we will arrive at a limiting number $x_{\infty} \in \Lambda$, with $\Phi(x_{\infty}) < \infty$ and such that none of the oracle TMs $M^{\phi}_{i}$ in our list properly compute $x_{\infty}$.

As in the proof of Theorem 4.3, we will need to carefully control the value of $\Phi(x)$ from the symbolic expansion of $x$. For the proof of Theorem~\ref{thm:non_comp_exists} we only need Lemma~\ref{lem:5.18} from the previous section, which is proven in the appendix under the weaker assumptions on $\Phi$ used in this section.

For the below proofs, we will say $\Phi(x)$ is computable at $x$ if there exists a Turing Machine $M^{\phi}$ such that if $\phi$ is an oracle for $x$, then on input $n$, $M^{\phi}$ outputs some $y'$ for which $|\Phi(x)-y'| \leq 2^{-n}$. This definition uses "$\leq$" instead of the "$<$" which is used in the definition given in section \ref{section:Preliminaries}, but it is easy to see that the two definitions are equivalent.

Before starting the proof, we need the following elementary fact.

\begin{lem}\label{cont_frac_continuity}
Write any number in $\Lambda$ as $\omega = [a_{1}, a_{2}, \ldots]$. For any $\varepsilon > 0$, there is an $L>0$ for which $n>L$ implies that for any sequence of natural numbers $(N_{0}, N_{1}, \ldots)$,
\begin{equation*}
    |\omega - [a_{1}, a_{2}, \ldots, a_{n-1}, N_{0}, N_{1}, \ldots]| < \varepsilon.
\end{equation*}
\end{lem}

\begin{proof}
Let $M>0$ be large enough that $\dfrac{s_{1}-s_{0}}{\tau^{M}} < \varepsilon$ and let $L = \kappa \cdot M$, where $\kappa, \tau$ are from assumption \ref{G:expansivity} on $G$. Noting that
\begin{equation*}
    \omega, [a_{1}, a_{2}, \ldots, a_{n-1}, N_{0}, N_{1}] \in (G_{a_{1}}^{-1} \circ G_{a_{2}}^{-1} \circ \cdots \circ G_{a_{n-1}}^{-1})((s_{0}, s_{1}))
\end{equation*}
and $|(G_{a_{i+\kappa-1}} \circ G_{a_{i+\kappa-2}} \circ \cdots \circ G_{a_{i}})'(\theta)| > \tau > 1 \implies |(G_{a_{i}}^{-1} \circ G_{a_{i+1}}^{-1} \circ \cdots \circ G_{a_{i+\kappa-1}}^{-1})'(\theta)| < \frac{1}{\tau} < 1$, as well as $|(G_{a_{i}}^{1})'(\theta)| < 1$ from \ref{G:expansivity}, we have
\begin{align*}
    &\text{length}((G_{a_{1}}^{-1} \circ \cdots \circ G_{a_{n-1}}^{-1})((s_{0}, s_{1}))) \\
    &\leq \left( \prod_{i=0}^{M-1} \sup_{\theta \in \Lambda} |(G_{a_{i\kappa+1}}^{-1} \circ \cdots \circ G_{a_{i\kappa+\kappa}}^{-1})'(\theta)| \right) \cdot \left( \prod_{j=L+1}^{n-1} \sup_{\theta \in \Lambda} |(G_{a_{j}}^{-1})'(\theta)| \right) \cdot (s_{1} - s_{0}) < \dfrac{s_{1}-s_{0}}{\tau^{M}} < \varepsilon.
\end{align*}

\end{proof}

In particular, we have the following:

\begin{cor}\label{cor:cont_frac_continuity_weak}
For $\omega = [a_{1}, a_{2}, \ldots]$ as above, for any $\varepsilon > 0$, there is an $L>0$ for which $n>L$ implies that $\forall N \in \mathbb{N}$,
\begin{equation*}
    |\omega - [a_{1}, a_{2}, \ldots, a_{n-1}, N, a_{n+1}, \ldots]| < \varepsilon.
\end{equation*}
\end{cor}

Before proceeding to the proof of the main result, we first define some notation. For any $x_{i} = [a_{1}^{i}, a_{2}^{i}, \ldots]$ let $\eta_{k}^{i}=[a_{k}^{i}, a_{k+1}^{i}, \ldots]$, noting that
\begin{equation*}
    \Phi(x_{i})=\sum_{n=1}^{\infty} s(i) \cdot \left( \eta_{0}^{i} \eta_{1}^{i} \cdots \eta_{n-1}^{i} \right)^{\nu} \cdot u(\eta_{n}^{i})
\end{equation*}
where $\nu, \rho >0$ and $s(i) \in \{-1, 1\}$ with $s(i) = 1$ infinitely often. Let
\begin{equation*}
    f(i,k)=\sum_{n=1}^{k}s(i) \cdot \left( \eta_{0}^{i} \eta_{1}^{i} \cdots \eta_{n-1}^{i} \right)^{\nu} \cdot u(\eta_{n}^{i}),
\end{equation*}
noting that $\lim_{k \to \infty}f(i, k) = \Phi(x_{i})$.

\subsection{Proof of Theorem~\ref{thm:non_comp_exists}}

We will first show inductively that there exist:

\begin{itemize}
    \item nested initial segments $I_{1} \subseteq I_{2} \subseteq \ldots$, where each $I_{i}$ has length $p_{i}$;
    \item for each $i = 1, 2, \ldots$, positive integers $N^{i}$ and $m_{i}$;
    \item positive integers $k_{1} < k_{2} < \cdots$ and $l_{1} < l_{2} < \cdots$ and $\hat{a}_{1} < \hat{a}_{2} < \cdots$ and $n(0) < n(1) < n(2) < \cdots$, positive real numbers $\hat{\varepsilon}_{1} > \hat{\varepsilon}_{2} > \cdots$, and oracles $\phi_{0}, \phi_{1}, \ldots$;
\end{itemize}

such that if we let $x_{i} = [I_{i}, 1, \ldots, 1, N^{i}, 1, \ldots]$ for $i \in \mathbb{Z}^{+}$, where $N^{i}$ is in the $(m_{i}+p_{i})$-th position, then we have the following:

\begin{enumerate}[label={(\arabic*)}]
    \item $\phi_{i}$ is an oracle for $x_{i}$ such that $|\phi_{i}(n) - x_{i}| < 2^{-(n+1)}$ for all $n$. \label{general_item1}
    \item $\phi_{i}$ agrees with the oracle $\phi_{j-1}$ on inputs $1, 2, \ldots, k_{j}$ for $j=1, 2, \ldots, i$. \label{general_item2}
    \item Running $M_{l_{i}}^{\phi_{i}}(1), M_{l_{i}}^{\phi_{i}}(2), \ldots, M_{l_{i}}^{\phi_{i}}(\hat{a}_{i})$ queries $\phi_{i}$ with parameters not exceeding $k_{i}$. \label{general_item4}
    \item Running $M_{l_{i}}^{\phi_{i}}(\hat{a}_{i})$ yields a number $A_{l_{i}}$ for which \label{general_item5}
    \begin{equation*}
        A_{l_{i}} + 2^{-\hat{a}_{i}} \leq \Phi(x_{i-1}) + 2^{-\hat{a}_{i}+1} < \Phi(x_{i}) < \Phi(x_{i-1}) + 2 \cdot 2^{-\hat{a}_{i}+1}.
    \end{equation*}
    \item Running $M_{l_{j}}^{\phi_{i}}(\hat{a}_{j})$ for $j = 1, 2, \ldots, i$ yields a number $B_{l_{j}}$ for which $B_{l_{j}} + 2^{-\hat{a}_{j}} < \Phi(x_{i})$. \label{general_item6}
    \item The TMs $M^{\phi_{i}}_{k}$ for $k \in \{1, \ldots, l_{i}\}$ all do not properly compute $\Phi(x_{i})$; in particular, they all compute $\Phi(x_{i})$ with an error of at least $\hat{\varepsilon}_{i}$. \label{general_item7}
    \item For $k\geq n(i)$, we have $|f(i,k)-\Phi(x_{i})| < 2^{-i}$. \label{general_item8}
    \item For $k = 1, 2, \ldots, n(i-1)$, we have $|f(i,k)-f(i-1,k)|<2^{-i}$. \label{general_item9}
\end{enumerate}

\subsubsection{Base case}

There are countably many oracle Turing Machines $M^{\phi}$, where $\phi$ represents an oracle for $x$, so we can order them as $M_{1}^{\phi}, M_{2}^{\phi}, \ldots$. Let $x_{0} = [1, 1, 1, \ldots]$. Given an oracle $\phi_{0}$ for $x_{0}$ such that $|\phi_{0}(n') - x_{0}| < 2^{-(n'+1)}$ for all $n'$, let $M_{l_{1}}^{\phi_{0}}$ be the first TM to compute $\Phi(x_{0})$ (if no such TM exists, we are done). Since this is the first such TM, all of $M_{1}^{\phi_{0}}, M_{2}^{\phi_{0}}, \ldots, M_{l_{1}-1}^{\phi_{0}}$ do not properly compute $\Phi(x_{0})$, so there are integers $a_{1}, a_{2}, \ldots, a_{l_{1}-1}$ and small positive real numbers $\varepsilon_{1}, \varepsilon_{2}, \ldots, \varepsilon_{l_{1}-1}$ for which $M_{k}^{\phi_{0}}(a_{k})$ outputs some number $A_{k}$ with $|A_{k} - \Phi(x_{0})| > 2^{-a_{k}} + \varepsilon_{k}$. Set $\hat{\varepsilon}_{1} = \min(\varepsilon_{1}, \ldots, \varepsilon_{l_{1}-1})$; choose $\hat{n}_{1}$ large enough so that $2^{-\hat{n}_{1}+2} < \hat{\varepsilon}_{1}/2$, and set $\hat{a}_{1} = \max(a_{1}, \ldots, a_{l_{1}-1}, \hat{n}_{1})$.

Run $M_{l_{1}}^{\phi_{0}}(\hat{a}_{1})$ with the oracle $\phi_{0}$. This TM outputs a number $A_{l_{1}}$ for which $|A_{l_{1}} - \Phi(x_{0})| \leq 2^{-\hat{a}_{1}}$. Since the computation is performed in finite time, there is a $k_{1}>0$ such that $\phi_{0}$ is only queried with parameters not exceeding $k_{1}$. We assume $k_{1}$ is large enough that the computations $M_{l_{1}}^{\phi_{0}}(k)$ query $\phi_{0}$ for parameters not exceeding $k_{1}$ for $k = 1, 2, \ldots, \hat{a}_{1}-1$. Setting $n(0) = 1$, we additionally make $k_{1}$ large enough such that for any $x_{1}$ with $|x_{1} - x_{0}| < 2^{-(k_{1}+1)}$ we have
\begin{equation*}
    |f(1, 1)-f(0, 1)| = \left| u(\eta_{1}^{1}) - u(\eta_{1}^{0}) \right| = \left| u(x_{1}) - u(x_{0}) \right| < 2^{-1}
\end{equation*}
by continuity of $u(\cdot)$. Hence \ref{general_item9} is satisfied.

Now for any $x_{1}$ such that $|x_{1} - x_{0}| < 2^{-(k_{1}+1)}$, $\phi_{0}$ is a valid oracle for $x_{1}$ up to parameter value $k_{1}$. In particular, we can create an oracle $\psi$ for $x_{1}$ which agrees with $\phi_{0}$ on $1, 2, \ldots, k_{1}$. Then the execution of $M_{l_{1}}^{\phi_{0}}(\hat{a}_{1})$ will be identical to that of $M_{l_{1}}^{\psi}(\hat{a}_{1})$, so it will output the same value $A_{l_{1}}$ which is a $2^{-\hat{a}_{1}}$-approximation for $\Phi(x_{0})$.


Applying Corollary~\ref{cor:cont_frac_continuity_weak} with $\varepsilon = 2^{-(k_{1} +1)}$, we get $L_{1}>0$ such that
\begin{equation*}
    \forall m_{1}>L_{1}, \forall N_{1} \in \mathbb{N}, |\beta^{N_{1}} - x_{0}| < 2^{-(k_{1}+1)},
\end{equation*}
where $\beta^{N_{1}}$ has all ones except an $N_{1}$ at the $(m_{1}+1)$-th position. Applying Lemma~\ref{lem:5.18} with $I_{1} = [1]$ and $\varepsilon = 2^{-\hat{a}_{1}+1}$, and making sure the integer $m=m_{1}$ we get from this Lemma satisfies $m_{1}>L_{1}$, we get some $\beta^{N_{1}} = \beta_{1}^{N_{1}}$ for which
\begin{equation*}
    \exists N_{1} \in \mathbb{N} \text{ such that } |\beta_{1}^{N_{1}} - x_{0}| < 2^{-(k_{1}+1)} \hspace{10px}[\text{since } m_{1}>L_{1}]
\end{equation*}
yet
\begin{equation*}
    \Phi(x_{0}) + 2^{-\hat{a}_{1}+1} < \Phi(\beta_{1}^{N_{1}}) < \Phi(x_{0}) + 2 \cdot 2^{-\hat{a}_{1}+1},
\end{equation*}
where $\beta_{1}^{N_{1}}$ has all ones except an $N_{1}$ at the $(m_{1}+1)$-th position. Let $x_{1} = \beta_{1}^{N_{1}}$, and let $\phi_{1} = \psi$ be the oracle for $x_{1}$ which agrees with $\phi_{0}$ on $1, 2, \ldots, k_{1}$ and which additionally satisfies $|\phi_{1}(n)-x_{1}|<2^{-(n+1)}$ for all $n$. This oracle $\phi_{1}$ satisfies \ref{general_item1} and \ref{general_item2}. As previously stated, the execution of $M_{l_{1}}^{\phi_{0}}(\hat{a}_{1})$ is identical to that of $M_{l_{1}}^{\phi_{1}}(\hat{a}_{1})$, so the output will be the same number $A_{l_{1}}$ such that $|A_{l_{1}} - \Phi(x_{0})| \leq 2^{-\hat{a}_{1}}$. By construction, $M_{l_{1}}^{\phi_{1}}(1), M_{l_{1}}^{\phi_{1}}(2), \ldots, M_{l_{1}}^{\phi_{1}}(\hat{a}_{1})$ only query $\phi_{1}$ with parameters not exceeding $k_{1}$, satisfying \ref{general_item4}. But then by the above work 
\begin{equation*}
    A_{l_{1}} + 2^{-\hat{a}_{1}} \leq \Phi(x_{0}) + 2^{-\hat{a}_{1}} + 2^{-\hat{a}_{1}} = \Phi(x_{0}) + 2^{-\hat{a}_{1}+1} < \Phi(x_{1}),
\end{equation*}
satisfying \ref{general_item5} and also \ref{general_item6} by taking $B_{l_{1}} = A_{l_{1}}$. Thus $M_{l_{1}}^{\phi_{1}}(\hat{a}_{1})$, where $\phi_{1}$ is an oracle for $x_{1}$, does not approximate $\Phi(x_{1})$ with precision $2^{-\hat{a}_{1}}$, so this TM does not properly compute this number. Additionally, note that $M_{k}^{\phi_{1}}$ also does not compute $\Phi(x_{1})$ for $k = 1, 2, \ldots, l_{1}-1$. Running $M_{k}^{\phi_{1}}(a_{k})$ is identical to running $M_{k}^{\phi_{0}}(a_{k})$ by our choice of $k_{1}$ in the construction of $\phi_{1}$, so $M_{k}^{\phi_{1}}(a_k)$ outputs a number $B_{k}$ for which $|B_{k} - \Phi(x_{0})| > 2^{-a_{k}} + \varepsilon_{k}$. Since 
\begin{equation*}
    |\Phi(x_{1})-\Phi(x_{0})| < 2^{-\hat{a}_{1}+2} < \frac{\hat{\varepsilon}_{1}}{2} \leq \frac{\varepsilon_{k}}{2},
\end{equation*}
we have
\begin{equation*}
    |B_{k}-\Phi(x_{1})| \geq |B_{k}-\Phi(x_{0})| - |\Phi(x_{0})-\Phi(x_{1})| > (2^{-a_{k}} + \varepsilon_{k}) - \frac{\varepsilon_{k}}{2} = 2^{-a_{k}} + \frac{\varepsilon_{k}}{2} > 2^{-a_{k}},
\end{equation*}
satisfying \ref{general_item7}. To show \ref{general_item8}, note that $\lim_{k \to \infty}f(1,k) = \Phi(x_{1})$ is finite, so there is some large enough $n(1)$ for which $k \geq n(1)$ implies $|f(1, k) - \Phi(x_{1})| < 2^{-1}$ as required.

\subsubsection{Induction step}

Now inductively, suppose there exist the following:

\begin{itemize}
    \item nested initial segments $I_{1} \subseteq \ldots \subseteq I_{i-1}$, where each $I_{j}$ has length $p_{j}$;
    \item for each $j = 1, 2, \ldots, i-1$, positive integers $N^{j}$ and $m_{j}$;
    \item positive integers $k_{1} < k_{2} < \cdots < k_{i-1}$ and $l_{1} < l_{2} < \cdots < l_{i-1}$ and $\hat{a}_{1} < \hat{a}_{2} < \cdots < \hat{a}_{i-1}$ and $n(0) < n(1) < \cdots < n(i-1)$, positive real numbers $\hat{\varepsilon}_{1} > \hat{\varepsilon}_{2} > \cdots > \hat{\varepsilon}_{i-1}$, and oracles $\phi_{0}, \phi_{1}, \ldots, \phi_{i-1}$;
\end{itemize}

such that if we let $x_{j} = [I_{j}, 1, \ldots, 1, N^{j}, 1, \ldots]$ for $j = 1, 2, \ldots, i-1$, where $N^{j}$ is in the $(m_{j}+p_{j})$-th position, then we have the following:

\begin{enumerate}[label={(\arabic*)}]
    \item $\phi_{j}$ is an oracle for $x_{j}$ such that $|\phi_{j}(n) - x_{j}| < 2^{-(n+1)}$ for all $n$. \label{IH_item1}
    \item $\phi_{i-1}$ agrees with the oracle $\phi_{j-1}$ on inputs $1, 2, \ldots, k_{j}$. \label{IH_item2}
    \item Running $M_{l_{j}}^{\phi_{j}}(1), M_{l_{j}}^{\phi_{j}}(2), \ldots, M_{l_{j}}^{\phi_{j}}(\hat{a}_{j})$ queries $\phi_{j}$ with parameters not exceeding $k_{j}$. \label{IH_item4}
    \item Running $M_{l_{j}}^{\phi_{j}}(\hat{a}_{j})$ yields a number $A_{l_{j}}$ for which \label{IH_item5}
    \begin{equation*} 
        A_{l_{j}} + 2^{-\hat{a}_{j}} \leq \Phi(x_{j-1}) + 2^{-\hat{a}_{j}+1} < \Phi(x_{j}) < \Phi(x_{j-1}) + 2 \cdot 2^{-\hat{a}_{j}+1}.
    \end{equation*}
    \item Running $M_{l_{j}}^{\phi_{i-1}}(\hat{a}_{j})$ yields a number $B_{l_{j}}$ for which $B_{l_{j}} + 2^{-\hat{a}_{j}} < \Phi(x_{i-1})$. \label{IH_item6}
    \item The TMs $M^{\phi_{i-1}}_{k}$ for $k \in \{1, \ldots, l_{i-1}\}$ all do not properly compute $\Phi(x_{i-1})$; in particular, they all compute $\Phi(x_{i-1})$ with an error of at least $\hat{\varepsilon}_{i-1}$. \label{IH_item7}
    \item For $k\geq n(i-1)$, we have $|f(i-1,k)-\Phi(x_{i-1})| < 2^{-(i-1)}$. \label{IH_item8}
    \item For $k = 1, 2, \ldots, n(i-2)$, we have $|f(i-1,k)-f(i-2,k)|<2^{-(i-1)}$. \label{IH_item9}
\end{enumerate}

Let $M_{l_{i}}^{\phi_{i-1}}$ be the first TM, with some $l_{i}>l_{i-1}$ and with $\phi_{i-1}$ being the oracle for $x_{i-1}$ from the induction hypothesis, which computes $\Phi(x_{i-1})$. We want to find an initial segment $I_{i} \supseteq I_{i-1}$ of some length $p_{i}$, positive integers $N^{i}, m_{i}, k_{i} > k_{i-1}, \hat{a}_{i} > \hat{a}_{i-1}, n(i)>n(i-1)$, and an oracle $\phi_{i}$ for $x_{i}$ such that \ref{IH_item1}-\ref{IH_item9} are satisfied for $i$ instead of $i-1$. It would follow from \ref{IH_item6} and \ref{IH_item7} that none of $M_{1}^{\phi_{i}}, M_{2}^{\phi_{i}}, \ldots, M_{l_{i}}^{\phi_{i}}$ properly compute  $\Phi(x_{i})$.

Since $M_{l_{i}}^{\phi_{i-1}}$ is the first such TM with $l_{i}>l_{i-1}$, combined with \ref{IH_item7} from the induction hypothesis we get that all of $M_{1}^{\phi_{i-1}}, M_{2}^{\phi_{i-1}}, \ldots, M_{l_{i}-1}^{\phi_{i-1}}$ do not properly compute $\Phi(x_{i-1})$. So, there are integers $a_{1}, a_{2}, \ldots, a_{l_{i}-1}$ and small positive real numbers $\varepsilon_{1}, \varepsilon_{2}, \ldots, \varepsilon_{l_{i}-1} < \hat{\varepsilon}_{i-1}$ for which $M_{k}^{\phi_{i-1}}(a_{k})$ outputs some number $A_{k}$ with $|A_{k} - \Phi(x_{i-1})| > 2^{-a_{k}} + \varepsilon_{k}$. Set $\hat{\varepsilon}_{i} = \min(\varepsilon_{1}, \ldots, \varepsilon_{l_{i}-1})$, choose $\hat{n}_{i}$ large enough so that $2^{-\hat{n}_{i}+2} < \hat{\varepsilon}_{i}/2^{i}$, and set \\
$\hat{a}_{i} = \max(a_{1}, \ldots, a_{l_{i}-1}, \hat{a}_{i-1}, \hat{n}_{i})$.

When run, the TM $M_{l_{i}}^{\phi_{i-1}}(\hat{a}_{i})$ outputs a number $A_{l_{i}}$ for which $|A_{l_{i}} - \Phi(x_{i-1})| \leq 2^{-\hat{a}_{i}}$. The computations $M_{l_{j}}^{\phi_{j}}(1), M_{l_{j}}^{\phi_{j}}(2), \ldots, M_{l_{j}}^{\phi_{j}}(\hat{a}_{i})$ are performed in finite time, so there is a $k_{i}>0$ such that $\phi_{i-1}$ is only queried with parameters not exceeding $k_{i}$ for all of these computations; we can make $k_{i}$ arbitrarily large, so assume $k_{i} > k_{i-1}$. We additionally make $k_{i}$ large enough such that for any $x_{i}$ with $|x_{i} - x_{i-1}| < 2^{-(k_{i}+1)}$ and any $k = 1, 2, \ldots, n(i-1)$, we have
\begin{align*}
    &|f(i, k)-f(i-1, k)| \\
    &= \left| \sum_{n=1}^{k} s(i) \cdot \left(\eta_{0}^{i} \eta_{1}^{i} \cdots \eta_{n-1}^{i}\right)^{\nu} \cdot u(\eta_{n}^{i}) - \sum_{n=1}^{k} s(i) \cdot \left(\eta_{0}^{i-1} \eta_{1}^{i-1} \cdots \eta_{n-1}^{i-1} \right)^{\nu} \cdot u(\eta_{n}^{i-1}) \right| < 2^{-i}.
\end{align*}
We can do this because the sums are finite and have continuous dependence on $x_{i}$, since $u$ and $G$ are $\mathcal{C}^{1}$ on the set $S$ (from the definition of $G$). Hence \ref{IH_item9} is satisfied.

Now, for any $x_{i}$ such that $|x_{i} - x_{i-1}| < 2^{-(k_{i}+1)}$, $\phi_{i-1}$ is a valid oracle for $x_{i}$ up to parameter value $k_{i}$. In particular, we can create an oracle $\psi$ for $x_{i}$ which agrees with $\phi_{i-1}$ on $1, 2, \ldots, k_{i}$ and so that $|\psi(n)-x_{i}|<2^{-(n+1)}$ for all $n$. Then the execution of $M_{l_{i}}^{\phi_{i-1}}(\hat{a}_{i})$ will be identical to that of $M_{l_{i}}^{\psi}(\hat{a}_{i})$, so it will output the same number $A_{l_{i}}$ which is a $2^{-\hat{a}_{i}}$-approximation for $\Phi(x_{i-1})$.

Applying Corollary~\ref{cor:cont_frac_continuity_weak} with $\varepsilon = 2^{-(k_{i} +1)}$, we get $L_{i}>0$ such that
\begin{equation*}
    \forall m_{i}>L_{i}, \forall N_{i} \in \mathbb{N}, |\beta^{N_{i}} - x_{i-1}| < 2^{-(k_{i}+1)},
\end{equation*}
where $\beta^{N_{i}} = [I_{i-1}, 1, \ldots, 1, N_{i-1}, 1, \ldots]$ agrees with $x_{i-1}$ at all entries except having an $N_{i-1}$ instead of a $1$ at the $(m_{i-1} + p_{i-1})$-th position. Applying Lemma~\ref{lem:5.18} with $I_{i} = [I_{i-1}, 1, \ldots, 1, N_{i-1}]$ (where $N_{i-1}$ is at the $(m_{i-1} + p_{i-1})$-th position) and $\varepsilon = 2^{-\hat{a}_{i}+1}$, and making sure the integer $m=m_{i}$ we get from this Lemma satisfies $m_{i}>L_{i}$, we get some $\beta^{N_{i}} = \beta_{i}^{N_{i}}$ for which
\begin{equation*}
    \exists N_{i} \in \mathbb{N} \text{ such that } |\beta_{i}^{N_{i}} - x_{i-1}| < 2^{-(k_{i}+1)} \hspace{10px}[\text{since } m_{i}>L_{i}]
\end{equation*}
yet
\begin{equation*}
    \Phi(x_{i-1}) + 2^{-\hat{a}_{i}+1} < \Phi(\beta_{i}^{N_{i}}) < \Phi(x_{i-1}) + 2 \cdot 2^{-\hat{a}_{i}+1}.
\end{equation*}
Let $x_{i} = \beta_{i}^{N_{i}}$, and let $\phi_{i} = \psi$ be the oracle for $x_{i}$ which agrees with $\phi_{i-1}$ on $1, 2, \ldots, k_{i}$. This $\phi_{i}$ satisfies \ref{IH_item1} and \ref{IH_item2}. As previously stated, the execution of $M_{l_{i}}^{\phi_{i-1}}(\hat{a}_{i})$ is identical to that of $M_{l_{i}}^{\phi_{i}}(\hat{a}_{i})$, so the output will be the same number $A_{l_{i}}$ such that $|A_{l_{i}} - \Phi(x_{i-1})| \leq 2^{-\hat{a}_{i}}$. By construction, $M_{l_{i}}^{\phi_{i}}(1), M_{l_{i}}^{\phi_{i}}(2), \ldots, M_{l_{i}}^{\phi_{i}}(\hat{a}_{i})$ only query $\phi_{i}$ with parameters not exceeding $k_{i}$, satisfying \ref{IH_item4}. But then by the above work we have 
\begin{equation*}
    A_{l_{i}} + 2^{-\hat{a}_{i}} \leq \Phi(x_{i-1}) + 2^{-\hat{a}_{i}} + 2^{-\hat{a}_{i}} = \Phi(x_{i-1}) + 2^{\hat{a}_{i}+1} < \Phi(x_{i}) < \Phi(x_{i-1}) + 2 \cdot 2^{-\hat{a}_{i}+1},
\end{equation*}
satisfying \ref{IH_item5}. Thus $M_{l_{i}}^{\phi_{i}}(\hat{a}_{i})$, where $\phi_{i}$ is an oracle for $x_{i}$, does not compute $\Phi(x_{i})$ with precision $2^{-\hat{a}_{i}}$, so it does not properly compute this number. But now since $\phi_{i}$ agrees with $\phi_{i-1}$ on $1, 2, \ldots, k_{i}$, by \ref{IH_item2} of the induction hypothesis we have that it agrees with $\phi_{j-1}$ on $1, 2, \ldots, k_{j}$ for $j = 1, 2, \ldots, i$. By \ref{IH_item4} of the induction hypothesis, running $M_{l_{j}}^{\phi_{j}}(\hat{a}_{j})$ queries $\phi_{j}$ with parameters not exceeding $k_{j} \leq k_{i}$, hence the execution of $M_{l_{j}}^{\phi_{j}}(\hat{a}_{j})$ is identical to that of $M_{l_{j}}^{\phi_{i}}(\hat{a}_{j})$ for all $j$. Therefore by \ref{IH_item5}, running $M_{l_{j}}^{\phi_{i}}(\hat{a}_{j})$ gives a number $B_{l_{j}} = A_{l_{j}}$ such that 
\begin{equation*}
    B_{l_{j}} + 2^{-\hat{a}_{j}} \leq \Phi(x_{j-1}) + 2^{\hat{a}_{j}+1} < \Phi(x_{j}) < \Phi(x_{j+1}) < \ldots < \Phi(x_{i}).
\end{equation*}
Hence \ref{IH_item6} is satisfied.

Now, note that $M_{k}^{\phi_{i}}$ also does not compute $\Phi(x_{i})$ properly for $k = 1, 2, \ldots, l_{i}-1$. Running $M_{k}^{\phi_{i}}(a_{k})$ is identical to running $M_{k}^{\phi_{i-1}}(a_{k})$ by our choice of $k_{i}$ in the construction of $\phi_{i}$, so $M_{k}^{\phi_{i}}(a_k)$ outputs a number $B_{k}$ for which $|B_{k} - \Phi(x_{i-1})| > 2^{-a_{k}} + \varepsilon_{k}$. Since 
\begin{equation*}
    |\Phi(x_{i})-\Phi(x_{i-1})| < 2^{-\hat{a}_{i}+2} < \frac{\hat{\varepsilon}_{i}}{2^{i}} \leq \frac{\varepsilon_{k}}{2},
\end{equation*}
we have
\begin{equation*}
    |A_{k}-\Phi(x_{i})| \geq |A_{k}-\Phi(x_{i-1})| - |\Phi(x_{i-1})-\Phi(x_{i})| > (2^{-a_{k}} + \varepsilon_{k}) - \frac{\varepsilon_{k}}{2} = 2^{-a_{k}} + \frac{\varepsilon_{k}}{2} > 2^{-a_{k}},
\end{equation*}
satisfying \ref{IH_item7}. Finally, it remains to show \ref{IH_item8} for $i$. Note that $\lim_{k \to \infty}f(i,k) = \Phi(x_{i})$ is finite, so there is some $n(i) > n(i-1)$ for which $k \geq n(i)$ implies $|f(i, k) - \Phi(x_{i})| < 2^{-i}$ as required. This completes the induction.

\subsubsection{Finalizing the argument}

Let $[a_{1}, a_{2}, \ldots] = x_{\infty} = \lim_{i \to \infty} x_{i}$. We claim that $\Phi(x_{\infty}) < \infty$ and $\Phi(x_{\infty})$ is not computable by any Turing Machine.

We first show that $\lim_{n \to \infty}\Phi(x_{i}) < \Phi(x_{0}) + 4 < \infty$. From \ref{general_item5} we have for all $i \in \mathbb{N}$ that
\begin{equation*}
    \Phi(x_{i}) < \Phi(x_{i-1}) + 2 \cdot 2^{-i+1},
\end{equation*}
hence
\begin{equation*}
    \Phi(x_{i}) < \Phi(x_{0}) + \sum_{j=1}^{i}2^{2} \cdot 2^{-j} \implies \sup_{i \in \mathbb{N}} \Phi(x_{i}) \leq \Phi(x_{0})+4.
\end{equation*}
We now show $\Phi(x_{\infty})$ is finite, and equals $\lim_{i\to\infty}\Phi(x_{i})$. For each $i$, the sequence $(f(j,n(i)))_{j=1}^{\infty}$ is Cauchy by \ref{general_item9}, hence it converges to some $f_{\infty}(n(i))$. It is clear that
\begin{equation*}
    f_{\infty}(n(i)) = \sum_{n=1}^{n(i)} s(i) \cdot \left( \eta_{0}^{\infty} \eta_{1}^{\infty} \cdots \eta_{n-1}^{\infty} \right)^{\nu} \cdot u(\eta_{n}^{\infty})
\end{equation*}
by definition of $x_{\infty}$ and by the way each successive $x_{j}$ was chosen. Since each $f(j, n(i)) \leq \Phi(x_{j})$, taking limits on both sides gives $f_{\infty}(n(i)) \leq \Phi(x_{0})+4$. Thus the sequence $(f_{\infty}(n(i)))_{i=1}^{\infty}$ is bounded from above, so the limit superior is finite. But now
\begin{equation*}
    \limsup_{i\to\infty}f_{\infty}(n(i)) = \limsup_{i\to\infty} \sum_{n=1}^{n(i)} s(i) \cdot \left( \eta_{0}^{\infty} \eta_{1}^{\infty} \cdots \eta_{n-1}^{\infty} \right)^{\nu} \cdot u(\eta_{n}^{\infty}) \geq \Phi(x_{\infty}),
\end{equation*}
showing that $\Phi(x_{\infty})$ is finite. Now we claim that $\Phi(x_{\infty}) = \lim_{i\to\infty}\Phi(x_{i})$. Let $\varepsilon>0$ be given.
\begin{itemize}
    \item Choose $i_{1}$ large enough so that \begin{equation*} i>i_{1} \implies |f_{\infty}(n(i))-\Phi(x_{\infty})| < \varepsilon/3. \end{equation*}
    \item Choose $i_{2}>i_{1}$ large enough so that $2^{-i_{2}}<\varepsilon/3$. By \ref{general_item8} we have \begin{equation*}
        i > i_{2} \implies |f(i, n(i))-\Phi(x_{i})| < \varepsilon/3.
    \end{equation*}
    \item Set $i_{3} = i_{2}+1$, so that $2^{-i_{3}+1} < \varepsilon/3$. Then for $i>i_{3}$, repeated application of \ref{general_item9} yields \begin{equation*}
        |f(i, n(i))-f_{\infty}(n(i))| \leq \sum_{j=i}^{\infty}|f(j+1, n(i))-f(j, n(i))| < \sum_{j=i}^{\infty}2^{-j} = 2^{-i+1}<\varepsilon/3.
    \end{equation*}
\end{itemize}
Taking $i>\max \{i_{1}, i_{2}, i_{3} \}$, we finally get
\begin{equation*}
    |\Phi(x_{\infty})-\Phi(x_{i})| \leq |\Phi(x_{\infty})-f_{\infty}(k)|+|f_{\infty}(n(i))-f(i,n(i))|+|f(i,n(i))-\Phi(x_{i})|< \varepsilon.
\end{equation*}
It remains to show $\Phi(x_{\infty})$ is not computable by any of the Turing machines $M^{\phi}_{i}$, where $\phi$ is an oracle for $x_{\infty}$. Let $\phi_{\infty}$ be the oracle for $x_{\infty}$ which agrees with $\phi_{i}$ on inputs $1, 2, \ldots, k_{i+1}$. This is a valid construction of an oracle by \ref{general_item2} and since $\lim_{i\to\infty}k_{i+1}=\infty$. To show $\phi_{\infty}$ is indeed an oracle for $x_{\infty}$, let $n \in \mathbb{Z}^{+}$ and set $i$ such that $n+1<k_{i}$. Since $n+1<k_{i}$ implies that $n+j+1<k_{i+j}$,
\begin{align*}
    |\phi_{\infty}(n)-x_{\infty}| & \leq |\phi_{\infty}(n)-x_{i}|+|x_{i}-x_{\infty}| \leq |\phi_{i}(n)-x_{i}| + \sum_{j=i}^{\infty}|x_{j+1}-x_{j}| \\
    &< 2^{-(n+1)} + \sum_{j=0}^{\infty}2^{-(k_{i+j}+1)} \leq 2^{-(n+1)} + \sum_{j=0}^{\infty}2^{-(n+j+2)} = 2^{-n},
\end{align*}
so $\phi_{\infty}$ is indeed an oracle for $x_{\infty}$. Now for a contradiction, suppose some TM $M_{j}^{\phi_{\infty}}$ computes $\Phi(x_{\infty})$. We have two cases:
\begin{enumerate}
\item $j = l_{i}$ for some $i$. Then $M_{l_{i}}^{\phi_{i}}(\hat{a}_{i})$ outputs $B_{i}$ for which $B_{i}+2^{-\hat{a}_{i}}<\Phi(x_{i})$ by \ref{general_item6}. This same number $B_{i}$ is output when running $M_{l_{i}}^{\phi_{\infty}}(\hat{a}_{i})$. But then \begin{equation*}
    B_{i}+2^{-\hat{a}_{i}}<\Phi(x_{i})<\Phi(x_{i+1})< \cdots < \Phi(x_{\infty}),
\end{equation*} so we cannot have $|B_{i}-\Phi(x_{\infty})| \leq 2^{-\hat{a}_{i}}$, contradicting the assumption that $M_{l_{i}}^{\phi_{\infty}}$ computes $\Phi(x_{\infty})$.
\item $j \neq l_{i}$ for all $i$. Choose the smallest $i>2$ for which $j<l_{i}$. Then $M_{j}^{\phi_{i-1}}(a_{j})$ outputs $A_{j}$ for which $|A_{j}-\Phi(x_{i-1})|>2^{-a_{j}}+\varepsilon_{j}$, where $a_{j}$ and $\varepsilon_{j}$ are from the $i$-th step of the induction. This same number $A_{j}$ is output when running $M_{j}^{\phi_{\infty}}(a_{j})$. By assumption, $|A_{j}-\Phi(x_{\infty})|<2^{-a_{j}}$, hence $|\Phi(x_{i-1})-\Phi(x_{\infty})|>\varepsilon_{j}$. But \begin{equation*}
    |\Phi(x_{k})-\Phi(x_{k+1})|<2^{-\hat{a}_{k}+2} \leq 2^{-\hat{n}_{k}+2}<\frac{\varepsilon_{j}}{2^{k}}
\end{equation*} for all $k \geq i-1$, thus \begin{equation*}
    |\Phi(x_{i-1})-\Phi(x_{\infty})| \leq \sum_{k=i-1}^{\infty}|\Phi(x_{k})-\Phi(x_{k+1})| < \sum_{k=i-1}^{\infty}\frac{\varepsilon_{j}}{2^{k}} \leq \varepsilon_{j},
\end{equation*} a contradiction.
\end{enumerate}

We have shown that none of the TMs $M_{j}^{\phi_{\infty}}$ compute $\Phi(x_{\infty})$, where $\phi_{\infty}$ is an oracle for $x_{\infty}$. This completes the proof of Theorem~\ref{thm:non_comp_exists}. \qed


\appendix

\section{Proofs of the main technical statements}
\label{section:Appendix}
\subsection{Proof of Theorem~\ref{thm:included}}
\label{sec-proof-included}
The most involved part of the proof will be showing that conditions (i)-(vii) on $G$ are satisfied by the $\alpha$-continued fraction expansion maps $A_{\alpha}$ for $\alpha \in [1/2, 1]$.

\begin{lem}  \label{lem:1a}       Conditions (i)-(vii) are satisfied by both the Gauss map and by $A_{\alpha}$ for $\alpha \in [1/2, 1)$.\end{lem}

\begin{proof}Properties (i), (iii), (v), (vi), (vii) are easy enough that we group them together and prove them at once, and then (ii) and (iv) are verified separately.

\textbf{(i), (iii), (v), (vi), (vii)}. For the Gauss map $G$, $G(J_{i}) = G( (\frac{1}{i+1}, \frac{1}{i}) ) = (0, 1) = (s_{0}, s_{1})$.  It is clear that $G$ is decreasing on $J_{1} = (1/2, 1)$. Since $J_{i} = (\ell_{i}, r_{i}) = (\frac{1}{i+1}, \frac{1}{i})$, we have
\begin{equation*}
    \dfrac{r_{i}-\ell_{i}}{\ell_{i}^{2}} < 2 =: D \text{\hspace{2px} and \hspace{2px}} g(i) = \frac{r_{i}}{\ell_{i+1}} \xrightarrow{i \to \infty} 0.
\end{equation*}
Finally, since $G$ is a composition of computable functions, it is computable.

Now let $\alpha \in [1/2, 1)$, and consider the map $A_{\alpha}$ corresponding to $\alpha$-continued fraction expansions. Set $n_{1} := \lceil 1/(1-\alpha) \rceil$; this is the smallest integer for which $1/n_{1} \leq 1-\alpha$ and $A_{\alpha}(1/n_{1}) = 0$. We restrict the domain and codomain of $A_{\alpha}(x)$ to the maximal invariant set $\Lambda \subseteq (s_{0}, s_{1}) := (0, 1/n_{1}) \subseteq (0, 1-\alpha)$, and henceforth consider $A_{\alpha}: \Lambda \to \Lambda$. It is clear that $\Lambda$ is a countable disjoint union of connected intervals, which we label $J_{1} = (\ell_{1}, r_{1}), J_{2} = (\ell_{2}, r_{2}), \ldots$ with $r_{1} > \ell_{1} \geq r_{2} > \ell_{2} \geq \cdots$. From maximality of the invariant set and the fact that $|A_{\alpha}'(x)|>1$ on $\Lambda$, (i) holds. For (vi), computability of $A_{\alpha}$ again follows since it a composition of computable functions. We will make use of the following facts about $A_{\alpha}$ and the intervals $J_{j}$:
\begin{itemize}
    \item $A_{\alpha}$ is decreasing on $J_{j}$ for odd $j$ and increasing on $J_{j}$ for even $j$. In particular, since $A_{\alpha}(r_{1}) = A_{\alpha}(1/n_{1}) = 0$, $|A_{\alpha}'(x)|>1$, and $A_{\alpha}(x) > 0$, $A_{\alpha}$ is decreasing on $J_{1}$ and so (iii) is satisfied.
    \item For $\alpha = 1/2$, it can be readily computed by solving the equations $A_{\alpha}(x) = 0$ and $A_{\alpha}(x) = 1$ that $n_{1} = 2$ and $J_{j} = (\ell_{j}, r_{j}) = (\frac{2}{j+4}, \frac{2}{j+3})$.
    \item For $\alpha \in (1/2, 1)$, we have
    \begin{align*}
        &\ell_{j} > \frac{2}{2(n_{1}-2+k)+3} = \frac{2}{j+2n_{1}}, \text{\hspace{10px}} r_{j} = \frac{1}{n_{1}-1+k} = \frac{2}{j+2n_{1}-1} \text{\hspace{5px} for $j = 2k-1$ odd, and} \\
        &\ell_{j} = \frac{1}{n_{1}+k} = \frac{2}{2n_{1}+j}, \text{\hspace{65px}} r_{j} < \frac{2}{2(n_{1}-2+k)+3} = \frac{2}{j+2n_{1}-1} \text{\hspace{5px} for $j = 2k$ even.}
    \end{align*}
\end{itemize}
It is easy to verify (v) and (vii) by showing that $\dfrac{r_{i}-\ell_{i}}{\ell_{i}^{2}}$ is bounded from above independently of $i$ and that $g(i) \xrightarrow{i \to \infty} 0$.

\textbf{(ii)}. For $G$, we note that $G'(x) = 1/x^{2}, G''(x) = 2/x^{3}$ where it is defined. Since $G$ is decreasing on all $J_{j}$, we have $\tau_{i,1} = |G'(r_{i})| = i^{2}$. Taking any $\sigma>2$, $\tau_{i,1}^{-1} < \ell_{i} \sigma$ as needed. Now set $\kappa = 2$. Then where it is defined, it is easy to check that $(G^{2})'' > 0$ and thus $\tau_{i,2} = |(G^{2})'(\ell_{i})| = (i+1)^{2}$. Hence for any $1 < \tau < 2$, we have $\tau_{i,2}^{-1} < \ell_{i} \cdot \tau^{-1}$ as needed.

Now again consider $A_{\alpha}$, for $\alpha \in [1/2, 1)$. As for the Gauss map, $A_{\alpha}'(x) = 1/x^{2}, A_{\alpha}''(x) = 2/x^{3} > 0$ where it is defined. Set $\kappa = 1$. Then $\tau_{i,\kappa} = \tau_{i, 1} = 1/r_{i}^{2}$.
\begin{itemize}
    \item Let $\sigma>1$ be large enough that for all $i$, $\sigma > \dfrac{2(2n_{1}+i)}{(i+2n_{1}-1)^{2}}$. It is easily verified that $\tau_{i,1}^{-1} < \ell_{i} \sigma$.
    \item Now, note that
    \begin{equation*}
        \tau_{i,1}^{-1} = r_{i}^{2} < \ell_{i} \tau^{-1} \iff \dfrac{2(2n_{1}+i)}{(i+2n_{1}-1)^{2}} < \tau^{-1}.
    \end{equation*}
    Since the left-hand side tends to zero as $n_{1} \to \infty$ for all $i \geq 1$, restricting $\Lambda$ to a small enough invariant set under $A_{\alpha}$, that is, making $r_{1} = \frac{1}{n_{1}}$ small enough, gives existence of some $\tau>1$ for which the above holds.
\end{itemize}

\textbf{(iv)}. For the Gauss map, recalling that $\varphi = \frac{\sqrt{5}-1}{2} = [1, 1, \ldots]$, we have
\begin{equation*}
    \delta_{G}(N) = \frac{1}{(N + \varphi)(N + 1 + \varphi)} \text{\hspace{5px} and so \hspace{5px}} \dfrac{r_{N+1}}{\ell_{N+1}} \cdot \dfrac{r_{n} - \ell_{N+1}}{\delta_{G}(N)} < D
\end{equation*}
for some constant $D$. We first show (iv) holds for $A_{1/2}$, then for $A_{\alpha}$, $\alpha \in (1/2, 1)$. We denote the symbolic expansion of a point in $\Lambda$ generated by $A_{1/2}$ as $[a_{1}, a_{2}, a_{3}, \ldots]_{1/2}$, and that of a point in the invariant set generated by $G$ as $[a_{1}, a_{2}, a_{3}, \ldots]_{1}$. Since the interval $J_{1}$ corresponding to $A_{1/2}$ is contained in the interval $J_{2}$ corresponding to $G$, we have $\psi := [1, 1, 1, \ldots]_{1/2} = [2, 2, 2, \ldots]_{1} = \sqrt{2}-1$. For $N = 2k-1$ odd, the interval $J_{N}$ corresponding to $A_{1/2}$ is contained in the interval $J_{(N+3)/2}$ corresponding to $G$, so
\begin{equation*}
    [N, 1, 1, \ldots]_{1/2} = [k+1, 2, 2, \ldots]_{1} = [(N+3)/2, 2, 2, \ldots]_{1} = \dfrac{1}{\frac{N+3}{2}+\psi} = \dfrac{2}{N + 3 + 2\psi}.
\end{equation*}
Now consider $N = 2k$ even. Observe that on $J_{N}, A_{1/2}(x) = 1 - G(x)$. We have
\begin{equation*}
    [N, 1, 1, \ldots]_{1/2} = (A_{1/2}|_{J_{N}})^{-1}([1, 1, \ldots]_{1/2}) = G_{N/2 + 1}^{-1}(1 - [2, 2, \ldots]_{1}) = \dfrac{2}{N+4-2\psi}.
\end{equation*}
Thus
\begin{align*}
    \delta_{A_{1/2}}(N) &= [N, 1, 1, \ldots]_{1/2} - [N+1, 1, 1, \ldots]_{1/2} = \dfrac{4 - 8\psi}{(N+3+2\psi)(N+5-2\psi)} \text{ if $N$ is odd, and} \\
    \delta_{A_{1/2}}(N) &= \dfrac{8\psi}{(N+4-2\psi)(N+4+2\psi)} \text{ if $N$ is even.}
\end{align*}
Noting that $r_{N} - \ell_{N+1} = \frac{4}{(N+3)(N+5)}$ and $\frac{r_{N+1}}{\ell_{N+1}} = \frac{N+5}{N+4}$, it is clear that some constant $D$ upper-bounds $\frac{r_{N+1}}{\ell_{N+1}} \cdot \frac{r_{N}-\ell_{N+1}}{\delta_{A_{1/2}}(N)}$.
Finally, for $\alpha \in (1/2, 1)$, we have $A_{1/2} = A_{\alpha}$ on the invariant set $\Lambda$ corresponding to $A_{\alpha}$. Thus the asymptotic behaviour of $\ell_{N}$, $r_{N}$, and $\delta_{A_{\alpha}}(N)$ is identical, so (iv) holds in this case as well.
\end{proof}

We can now prove Theorem~\ref{thm:included} without too much difficulty.

\textit{Proof of Theorem~\ref{thm:included}.} By Lemma~\ref{lem:1a} and since the function $u$ in the definition of $\cB_{\alpha,u,\nu}$ satisfies $\lim_{x \to 0^{+}}u(x) = \infty$ by assumption, conditions (i)-(vii) on $G$ and (i)-(v) on $u$ hold for $\cB_{\alpha, u, v}$. To prove this Theorem is remains to show that for $G = A_{\alpha}$, $\alpha \in [1/2,1]$, conditions (i)-(v) hold for both $u(x) = \log^{n}(1/x), n \in \mathbb{Z}^{+}$ and $u(x) = x^{-1}$. It is clear that both these functions are computable on $\mathbb{R}^{+}$ and tend to infinity near zero. If $\alpha \in [1/2, 1)$ then $s_{1}<1$, so \ref{u:2} in this case is satisfied. Thus we only need to show \ref{u:2} in the case that $G$ is the ordinary Gauss map, and \ref{u:3}.

For $u(x) = x^{-1}$, \ref{u:3} is obvious. If $u(x) = \log^{n}(1/x)$, then $|u'(x)| = \frac{n}{x} \cdot \log^{n-1}(1/x)$. Since there is some $C_{n}>0$ for which $\log^{n-1}(y) < C_{n}y$ for all large enough $y$, we have $\log^{n-1}(1/x) < \frac{C_{n}}{x}$ for all small enough $x$. Thus $|u'(x)| < \frac{nC_{n}}{x^{2}}$ for all small enough $x$ and since $|u'(x)|$ is bounded for $x$ bounded away from zero, there is some $C>0$ for which $|u'(x)| < \dfrac{C}{(x-s_{0})^{2}}$ for all $x \in (s_{0}, s_{1})$ as needed.

Now suppose $G$ is the Gauss map. Then $u \circ G_{1}^{-1} \circ G_{N}^{-1}(z)$ is decreasing with respect to $N$ since $u(\cdot)$ is decreasing, so
\begin{equation*}
    \frac{u \circ G_{1}^{-1} \circ G_{N}^{-1}(z)}{u \circ G_{1}^{-1} \circ G_{N}^{-1}(w)} \geq \frac{u \circ G_{1}^{-1} \circ G_{N}^{-1}(1)}{u \circ G_{1}^{-1} \circ G_{N}^{-1}(0)} =: v(N).
\end{equation*}
A tedious computation shows that $v'(N)>0$ and $v(N)>0$ for all $N$ and both $u(x) = \log^{n}(1/x)$ and $u(x) = x^{-1}$, therefore \ref{u:2} holds. \qed

\subsection{Proofs of the main lemmas}
\label{sec-proof-main-lemmas}
Here we will prove the three main Lemmas used in the above work. Lemma~\ref{lem:5.18} will be proven under the weaker assumptions on $G$ and $u$ along with the $s(i)$ term in section \ref{section:Generalized non-computability result}, and Lemmas \ref{lem:5.19} and \ref{lem:5.20} will be proven under the full assumptions in section \ref{section:Statements}. We recall these Lemmas below.

\textbf{Lemma~\ref{lem:5.18}.} \textit{For any initial segment $I = [a_{1}, a_{2}, \ldots, a_{n}]$, write $\omega = [a_{1}, a_{2}, \ldots, a_{n}, 1, 1, 1, \ldots]$. Then for any $\varepsilon>0$, there is an $m>0$ and an integer $N$ such that if we write $\beta = [a_{1}, a_{2}, \ldots, a_{n}, 1, 1, \ldots, 1, N, 1, 1, \ldots]$, where the $N$ is located in the $(n+m)$-th position, then}
\begin{equation*}
    \Phi(\omega) + \varepsilon < \Phi(\beta) < \Phi(\omega) + 2\varepsilon.
\end{equation*}

\textbf{Lemma~\ref{lem:5.19}.} \textit{Write $\omega = [a_{1}, a_{2}, \ldots, a_{n}, 1, 1, 1, \ldots]$. Then for any $\varepsilon>0$ there is an $m_{0}>0$, which can be computed from $(a_{1}, a_{2}, \ldots, a_{n})$ and $\varepsilon$, such that for any $m \geq m_{0}$ and for any tail $I = [a_{n+m}, a_{n+m+1}, \ldots]$,
\begin{equation*}
    \Phi(\beta^{I}) > \Phi(\omega) - \varepsilon
\end{equation*}
where}
\begin{equation*}
    \beta^{I} =  [a_{1}, a_{2}, \ldots, a_{n}, 1, 1, \ldots, 1, a_{n+m}, a_{n+m+1}, \ldots].
\end{equation*}

\textbf{Lemma~\ref{lem:5.20}.} \textit{Let $\omega = [a_{1}, a_{2}, \ldots]$ be such that $\Phi(\omega) < \infty$. Write $\omega_{k} = [a_{1}, a_{2}, \ldots, a_{k}, 1, 1, \ldots]$. Then for every $\varepsilon>0$ there is an $m$ such that, for all $k \geq m$,}
\begin{equation*}
    \Phi(\omega_{k}) < \Phi(\omega) + \varepsilon.
\end{equation*}

Write
\begin{equation*}
    \beta^{N} = [a_{1}, a_{2}, \ldots, a_{n}, 1, 1, \ldots, 1, N, 1, 1, \ldots],
\end{equation*}
where $N$ is in the $(m+n)$-th position. The following preliminary Lemmas are required for all of the main proofs below.

\begin{lem}\label{lem:u_bounded}
For any $s_{0}<a<s_{1}$, the function $u$ is bounded on the interval $[a, s_{1})$.
\end{lem}

\begin{proof}
By assumption on $u$, there is some $C>0$ for which $-u'(x) < \frac{C}{(x-s_{0})^{2}}$ for all $x \in [a, s_{1})$. Integrating both sides gives $u(x) < \frac{\tilde{C}}{x-s_{0}} \leq \frac{\tilde{C}}{a-s_{0}}$ for another constant $\tilde{C}>0$ and all $x \in [a, s_{1})$.
\end{proof}

\begin{lem}\label{lem:log_ratio_estimates}
Let $\beta^{N}, \beta^{1}$ be as before, let $\beta^{I}, \omega$ be as in Lemma~\ref{lem:5.19}, and let $\gamma_{1}, \gamma_{2}$ be two numbers whose symbolic representations coincide in the first $n+m-1$ terms $[a_{1}, a_{2}, \ldots, a_{n+m-1}]$ (in particular, we could have $\gamma_{1} = \beta^{N}$ and $\gamma_{2} = \beta^{1}$). Then for $g$ and $m_{g}$ as in \ref{G:function g}, for $\kappa$ as in \ref{G:expansivity}, and for any N, the following holds:

\begin{enumerate}[label={(\arabic*)}]
    \item \label{lem:u_diff_bounded:1}  For $i \leq n+m$, \begin{equation*}
        \left |  \log \frac{\eta_{i}(\beta^{N})}{\eta_{i}(\beta^{N+1})} \right | < \dfrac{g(N)}{\ell_{1}} \cdot \tau^{1+\frac{i-(n+m)}{\kappa}}.
    \end{equation*}
    \item \label{lem:u_diff_bounded:2} For $i < n+m$, \begin{equation*}
        \left |  \log \frac{\eta_{i}(\gamma_{1})}{\eta_{i}(\gamma_{2})} \right | < m_{g} \cdot \sigma \cdot \tau^{1 + \frac{i+1-(n+m)}{\kappa}}.
    \end{equation*}
    \item \label{lem:u_diff_bounded:3} For all large enough $m$, there is some positive function $f$ with $\lim_{m \to \infty}f(m) = 0$ such that for $i<n+m-1$, \begin{equation*}
        \left |  \log \frac{u(\eta_{i}(\beta^{I}))}{u(\eta_{i}(\omega))} \right | < f(m).
    \end{equation*}
    In particular, we could take $\beta^{I} = \beta^{1}$ and $\omega = \beta^{N}$. It clearly follows that $0 < \sup_{m \in \mathbb{N}}f(m) =: M < \infty$.
\end{enumerate}
\end{lem}

\begin{proof}[Proof of \ref{lem:u_diff_bounded:1}]
We will first show that for all $i$ we have
\begin{equation} \label{eq:1}
    |\eta_{i}(\beta^{N}) - \eta_{i}(\beta^{N+1})| < g(N) \cdot \tau^{1+\frac{i-(n+m)}{\kappa}}. \tag{$*_{1}$}
\end{equation}

If $i=n+m$, since $[N, 1, 1, \ldots] \in J_{N}$, $[N+1, 1, 1, \ldots] \in J_{N+1}$, for all $N \in \mathbb{N}$ we have
\begin{equation*}
    1 < \dfrac{\eta_{n+m}(\beta^{N})}{\eta_{n+m}(\beta^{N+1})} = \dfrac{[N, 1, 1, \ldots]}{[N+1, 1, 1, \ldots]} < \dfrac{r_{N}}{\ell_{N+1}} < 1 + g(N)
\end{equation*}
and so 
\begin{equation*}
    |\eta_{n+m}(\beta^{N}) - \eta_{n+m}(\beta^{N+1})| < g(N) < g(N) \cdot \tau^{1+\frac{i-(n+m)}{\kappa}}.
\end{equation*}

Now let $1 \leq j < \kappa$. We wish to show \ref{eq:1} for $i = n+m-j$. Note that since $|G_{1}'(x)| > 1$ for all $x \in S$ by \ref{G:expansivity} and $G_{1}$ is $\mathcal{C}^{1}$, the Inverse Function Theorem gives $|(G_{1}^{-1})'(x)| < 1$ for all $x \in J_{1}$. Repeatedly applying the Mean Value Theorem gives
\begin{equation*}
    |\eta_{n+m-j}(\beta^{N}) - \eta_{n+m-j}(\beta^{N+1})| < |\eta_{n+m}(\beta^{N}) - \eta_{n+m}(\beta^{N+1})| < g(N) < g(N) \cdot \tau^{1-j / \kappa}.
\end{equation*}
For $i = n+m - \kappa$, since $|(G_{1}^{\kappa})'|>\tau$ by assumption \ref{G:expansivity} on $G$, we similarly use Inverse Function Theorem and Mean Value Theorem (applied to the function $G_{1}^{-\kappa}$ this time) to get
\begin{equation*}
    |\eta_{n+m-\kappa}(\beta^{N}) - \eta_{n+m-\kappa}(\beta^{N+1})| < \tau^{-1} \cdot |\eta_{n+m}(\beta^{N}) - \eta_{n+m}(\beta^{N+1})| < g(N) \cdot \tau^{1+\frac{i-(n+m)}{\kappa}}
\end{equation*}
To show \ref{eq:1} for all other $i>n+m-j$ with $j>\kappa$, we apply the inequality
\begin{equation*}
    |\eta_{n+m-j}(\beta^{N}) - \eta_{n+m-j}(\beta^{N+1})| < \tau^{-1} \cdot |\eta_{n+m-j+\kappa}(\beta^{N}) - \eta_{n+m-j+\kappa}(\beta^{N+1})|
\end{equation*}
repeatedly until one of the cases $i \geq n+m-\kappa$ above is reached.

We now proceed to prove \ref{lem:u_diff_bounded:1}. If $i=n+m$ then as before for all $N \in \mathbb{Z}^{+}$ we have 
\begin{equation*}
    1 < \dfrac{\eta_{i}(\beta^{N})}{\eta_{i}(\beta^{N+1})} < 1 + g(N)
\end{equation*}
and so
\begin{equation*}
    \left| \log \dfrac{\eta_{i}(\beta^{N})}{\eta_{i}(\beta^{N+1})} \right| < \log \left( 1 + g(N) \right) < \log(e^{g(N)}) = g(N) < \dfrac{g(N)}{\ell_{1}} \cdot \tau^{1+\frac{i-(n+m)}{\kappa}},
\end{equation*}
where we used the inequality $1 + \xi < e^{\xi}$ for any $\xi>0$.

Let $1 \leq i < n+m$, and assume $\eta_{i}(\beta^{N}) > \eta_{i}(\beta^{N+1})$; the complementary case is almost identical. We have $\eta_{i}(\beta^{N}), \eta_{i}(\beta^{N+1}) > \ell_{1}$ since they are in $J_{1}$ and so from \ref{eq:1},
\begin{equation*}
    \left| \dfrac{\eta_{i}(\beta^{N}) - \eta_{i}(\beta^{N+1})}{\eta_{i}(\beta^{N+1})} \right| < \dfrac{g(N)}{\ell_{1}} \cdot \tau^{1+\frac{i-(n+m)} {\kappa}} \text{\hspace{5px} and so \hspace{5px}}1 < \dfrac{\eta_{i}(\beta^{N})}{\eta_{i}(\beta^{N+1})} < 1 + \dfrac{g(N)}{\ell_{1}} \cdot \tau^{1+\frac{i-(n+m)}{\kappa}}
\end{equation*}
As above, this gives the desired inequality.
\end{proof}

\begin{proof}[Proof of \ref{lem:u_diff_bounded:2}]
Note that $\eta_{i}(\gamma_{1})$ and $\eta_{i}(\gamma_{2})$ agree on the first digit of the symbolic expansion for each $1 \leq i < n+m$; we write this digit as $N_{i}$.

We will first show that for all $i$ we have
\begin{equation} \label{eq:2}
    \dfrac{\eta_{i}(\gamma_{1})}{\eta_{i}(\gamma_{2})} < 1 + m_{g} \cdot \sigma \cdot \tau^{1 + \frac{i+1-(n+m)}{\kappa}}. \tag{$*_{2}$}
\end{equation}
If $\eta_{i}(\gamma_{1}) < \eta_{i}(\gamma_{2})$ this is obviously true as the second term on the right hand side is greater than zero, but to unify the argument we will not split into cases.

For $i = n+m-1$, using assumption \ref{G:function g} on $G$ we compute
\begin{equation*}
    \dfrac{\eta_{n+m-1}(\gamma_{1})}{\eta_{n+m-1}(\gamma_{2})} < \dfrac{r_{N_{n+m-1}}}{\ell_{N_{n+m-1}}} < 1 + g(N_{n+m-1}) \leq 1 + m_{g} < 1 + m_{g} \cdot \sigma \cdot \tau^{1+\frac{i+1-(n+m)}{\kappa}}.
\end{equation*}

For future use, note that the above inequalities yield
\begin{equation*}
    \eta_{n+m-1}(\gamma_{1}) - \eta_{n+m-1}(\gamma_{2}) \leq \eta_{n+m-1}(\gamma_{2}) \cdot m_{g} < m_{g}.
\end{equation*}
Now let $1 \leq j < \kappa$. We want to show \ref{eq:2} for $i = n+m-1-j$. Similar to the proof of \ref{lem:u_diff_bounded:1}, we note that $|(G_{N_{n+m-1-j}}^{-1})'(x)| < \tau_{N_{n+m-1-j},1}^{-1}$ for all $x \in J_{N_{n+m-1-j}}$ and \ref{G:expansivity} gives $|(G_{N_{t}}^{-1})'(x)| < 1$ for all $x \in J_{N_{t}}$, $t$ arbitrary, so repeatedly applying the Mean Value Theorem,
\begin{align*}
    &\hspace{18px} \eta_{n+m-1-j}(\gamma_{1}) - \eta_{n+m-1-j}(\gamma_{2}) \\
    &< |G_{N_{n+m-1-j}}^{-1} \circ \cdots \circ G_{N_{n+m-2}}^{-1}(\eta_{n+m-1}(\gamma_{1})) - G_{N_{n+m-1-j}}^{-1} \circ \cdots \circ G_{N_{n+m-2}}^{-1}(\eta_{n+m-1}(\gamma_{2}))| \\
    &< \tau_{N_{n+m-1-j},1}^{-1} \cdot |G_{N_{n+m-j}}^{-1} \circ \cdots \circ G_{N_{n+m-2}}^{-1}(\eta_{n+m-1}(\gamma_{1})) - G_{N_{n+m-j}}^{-1} \circ \cdots \circ G_{N_{n+m-2}}^{-1}(\eta_{n+m-1}(\gamma_{2}))| \\
    &< \tau_{N_{n+m-1-j},1}^{-1} \cdot |\eta_{n+m-1}(\gamma_{1}) - \eta_{n+m-1}(\gamma_{2})| < \tau_{N_{n+m-1-j},1}^{-1} \cdot m_{g} \\
    &< \ell_{N_{n+m-1-j}} \cdot \sigma \cdot m_{g} < \eta_{n+m-1-j}(\gamma_{2}) \cdot \sigma \cdot m_{g} \cdot \tau^{1+\frac{i+1-(n+m)}{\kappa}}.
\end{align*}
Rearranging gives \ref{eq:2}. To show \ref{eq:2} holds for $i - l \cdot \kappa$ with $n+m-1-\kappa < i \leq n+m-1$ and $l \in \mathbb{Z}^{+}$, repeatedly apply the inequality
\begin{align*}
    &\hspace{18px} \eta_{i-j \cdot \kappa}(\gamma_{1}) - \eta_{i-l \cdot \kappa}(\gamma_{2}) \\
    &< |G_{N_{i-l \cdot \kappa}}^{-1} \circ \cdots \circ G_{N_{i-l \cdot \kappa + \kappa - 1}}^{-1}(\eta_{i-(l-1) \cdot \kappa}(\gamma_{1})) - G_{N_{i-l \cdot \kappa}}^{-1} \circ \cdots \circ G_{N_{i-l \cdot \kappa + \kappa - 1}}^{-1}(\eta_{i-(l-1) \cdot \kappa}(\gamma_{2}))| \\
    &= \left| \dfrac{1}{(G^{\kappa})'(\xi)} \right| \cdot |\eta_{i-(l-1) \cdot \kappa}(\gamma_{1}) - \eta_{i-(l-1) \cdot \kappa}(\gamma_{2})| < \tau^{-1} \cdot |\eta_{i-(l-1) \cdot \kappa}(\gamma_{1}) - \eta_{i-(l-1) \cdot \kappa}(\gamma_{2})|
\end{align*}
(where $\xi \in J_{N_{i-j \cdot \kappa}}$ is from the Mean Value Theorem and we made use of \ref{G:expansivity}) until one of the base cases above is reached.

We now suppose without loss of generality that $\eta_{i}(\gamma_{1}) \geq \eta_{i}(\gamma_{2})$. Using the inequality $1 + \xi < e^{\xi}$ for any $\xi>0$ with \ref{eq:2} gives
\begin{equation*}
    0 < \log \dfrac{\eta_{i}(\gamma_{1})}{\eta_{i}(\gamma_{2})} < \log \left( 1 + m_{g} \cdot \sigma \cdot \tau^{1 + \frac{i+1-(n+m)}{\kappa}} \right) < \log(e^{m_{g} \cdot \sigma \cdot \tau^{1 + \frac{i+1-(n+m)}{\kappa}}}) = m_{g} \cdot \sigma \cdot \tau^{1 + \frac{i+1-(n+m)}{\kappa}}.
\end{equation*}
On the other hand, the inequalities $\dfrac{1}{1-\xi} > e^{\xi}$ and $1 > \dfrac{\eta_{i}(\gamma_{2})}{\eta_{i}(\gamma_{1})} > \dfrac{1}{1+m_{g} \cdot \sigma \cdot \tau^{1 + \frac{i+1-(n+m)}{\kappa}}}$ give
\begin{equation*}
    0 > \log \dfrac{\eta_{i}(\gamma_{2})}{\eta_{i}(\gamma_{1})} > \log \left( \dfrac{1}{1+m_{g} \cdot \sigma \cdot \tau^{1 + \frac{i+1-(n+m)}{\kappa}}} \right) > \log(e^{-m_{g} \cdot \sigma \cdot \tau^{1 + \frac{i+1-(n+m)}{\kappa}}}) = -m_{g} \cdot \sigma \cdot \tau^{1 + \frac{i+1-(n+m)}{\kappa}}.
\end{equation*}
Combining these together yields
\begin{equation*}
    \left| \log \dfrac{\eta_{i}(\gamma_{1})}{\eta_{i}(\gamma_{2})} \right| < m_{g} \cdot \sigma \cdot \tau^{1 + \frac{i+1-(n+m)}{\kappa}}.
\end{equation*}
\end{proof}

\begin{proof}[Proof of \ref{lem:u_diff_bounded:3}]
Let $(\delta_{m})_{m=1}^{\infty}$ be a sequence of positive numbers such that $|\eta_{i}(\beta^{I})-\eta_{i}(\omega)| < \delta(m)$ for all $0 < i < n+m-1$ and $N \in \mathbb{N}$, and $\lim_{m \to \infty} \delta(m) = 0$. We claim $\eta_{i}(\beta^{I})$ is uniformly bounded away from $0$ and $1$ for all $I$ and $0<i<n+m-1$. Note that
\begin{align*}
    &\eta_{i}(\beta^{I}) = [1, 1, \ldots] \text{\hspace{50px} for $n<i<n+m-1$,} \\
    &\eta_{i}(\beta^{I}) = [a_{n}, 1, \ldots] \text{\hspace{44px} for $i=n$, and} \\
    &\eta_{i}(\beta^{I}) = [a_{i}, a_{i+1}, \ldots] \text{\hspace{32px} for $0<i<n$.}
\end{align*}
Call these first two digits $p_{i}$ and $q_{i}$, noting that $p_{i} = q_{i} = p_{n+1} = q_{n+1}$ if $n<i<n+m-1$. Let
\begin{align*}
    &\eta_{\min} = \min_{0 < i < n+m-1} [p_{i}, q_{i}, \ldots] = \min_{0 < i \leq n+1} [p_{i}, q_{i}, \ldots] = \min_{0 < i \leq n+1} (G_{p_{i}}^{-1} \circ G_{q_{i}}^{-1})((s_{0}, s_{1})), \\
    &\eta_{\max} = \max_{0 < i < n+m-1} [p_{i}, q_{i}, \ldots] = \max_{0 < i \leq n+1} [p_{i}, q_{i}, \ldots] = \max_{0 < i \leq n+1} (G_{p_{i}}^{-1} \circ G_{q_{i}}^{-1})((s_{0}, s_{1})).
\end{align*}
It is clear that if $p_{i} \neq 1$ or $q_{i} \neq 1$ then $s_{0} < \eta_{\min} < \eta_{\max} < s_{1}$. If $p_{i}=q_{i}=1$, by assumption \ref{G:decreasing} on $G$ we have
\begin{equation*}
    (G_{p_{i}}^{-1} \circ G_{q_{i}}^{-1})((s_{0}, s_{1})) = G_{p_{i}}^{-1}(G_{q_{i}}^{-1}(s_{1}), s_{1}) = (G_{q_{i}}^{-1}(s_{1}), (G_{p_{i}}^{-1} \circ G_{q_{i}}^{-1})(s_{1}))
\end{equation*}
with $s_{0} < G_{q_{i}}^{-1}(s_{1}) < (G_{p_{i}}^{-1} \circ G_{q_{i}}^{-1})(s_{1}) < s_{1}$. Hence $s_{0} < \eta_{\min} < \eta_{\max} < s_{1}$ in all cases. Thus
\begin{equation*}
    \eta_{i}(\beta^{I}) \in [\eta_{\min}, \eta_{\max}] \subseteq (s_{0}, s_{1}).
\end{equation*}
Clearly $\eta_{i}(\omega) \in [\eta_{\min}, \eta_{max}]$ as well. Now letting
\begin{equation*}
    \tilde{u} = \inf_{x \in [\eta_{\min}, \eta_{\max}]}u(x),
\end{equation*}
we have $\tilde{u}>0$ since $u$ is continuous and positive on the compact subset $[\eta_{\min}, \eta_{\max}]$.

Since $u$ is continuous, it is uniformly continuous on $[\eta_{\min}, \eta_{\max}]$, hence admits a continuous and strictly increasing modulus of continuity $\sigma: [0, \infty) \to [0, \infty)$ such that $\lim_{t \to 0} \sigma(t) = \sigma(0) = 0$ and for all $x_{1}, x_{2} \in [\eta_{\min}, \eta_{\max}], |u(x_{1})-u(x_{2})| \leq \sigma(|x_{1}-x_{2}|)$.

Now fix $m$ large enough that $1 - \dfrac{\sigma(\delta(m))}{\tilde{u}} > 0$, and let $0 < i < n+m-1$ and $N \in \mathbb{N}$. Then
\begin{align*}
    &|u(\eta_{i}(\beta^{I}))-u(\eta_{i}(\omega))| \leq \sigma(|\eta_{i}(\beta^{I})-\eta_{i}(\gamma_{2})|) < \sigma(\delta(m)) \\[0.4em]
    &\implies \left| \dfrac{u(\eta_{i}(\beta^{I}))}{u(\eta_{i}(\omega))} -1 \right| < \dfrac{\sigma(\delta(m))}{u(\eta_{i}(\omega))} \leq \dfrac{\sigma(\delta(m))}{\tilde{u}} \\[0.4em]
    &\implies \left| \log \left( \dfrac{u(\eta_{i}(\beta^{I}))}{u(\eta_{i}(\omega))} \right) \right| < \max \left\{ \left| \log \left( 1-\dfrac{\sigma(\delta(m))}{\tilde{u}} \right) \right|, \left| \log \left( 1+\dfrac{\sigma(\delta(m))}{\tilde{u}} \right)  \right| \right\} =: f(m),
\end{align*}
with 
\begin{equation*}
    \lim_{m \to \infty}f(m) = 0 \text{\hspace{10px} since \hspace{10px}} \lim_{\delta(m) \to 0^{+}} \log \left( 1 \pm  \dfrac{\sigma(\delta(m))}{\tilde{u}} \right) = 0
\end{equation*}
and $f(m)>0$ since $\dfrac{\sigma(\delta(m))}{\tilde{u}} > 0$ for all $m$.
\end{proof}

\begin{lem}\label{lem:consecutive_eta_bound}
There is some $\rho<1$ such that for any $k>1$ and $x \in \Lambda$, $\eta_{k-1}(x) \cdot \eta_{k}(x) < \rho$.
\end{lem}

\begin{proof}
If $s_{1}<1$ we can take $\rho = s_{1}$, so suppose $s_{1} = 1$. By the decreasing assumption \ref{G:decreasing} on $G_{1}$, it is easily checked that there are some $a_{*}, b_{*} \in (s_{0}, s_{1})$ for which $G(S \cap [a_{*}, 1)) \subseteq (s_{0}, b_{*})$.
\begin{itemize}
    \item If $G^{k-1}(x)<a_{*}$, then since $G^{k}(x) < 1$ we have $G^{k}(x) \cdot G^{k-1}(x) < a_{*} < 1$.
    \item If $a_{*} \leq G^{k-1}(x) < 1$, then $G^{k}(x) = G(G^{k-1}(x)) < b_{*}$ and so $G^{k}(x) \cdot G^{k-1}(x) < b_{*} < 1$.
\end{itemize}
Taking $\rho = \max\{a_{*}, b_{*}\}$ completes the proof.
\end{proof}

\subsection{Proof of Lemma~\ref{lem:5.19} and Lemma~\ref{lem:5.20}}

\begin{lem}\label{lem:u_diff_bounded}
There is a constant $D$ such that for any $\gamma_{1}, \gamma_{2} \in \Lambda$ with $\gamma_{1} = [a, \ldots]$, $\gamma_{2} = [a, \ldots]$ for some $a \in \mathbb{Z}^{+}$, we have
\begin{equation*}
    \left| u\left( \gamma_{1} \right) - u\left( \gamma_{2} \right) \right| < D.
\end{equation*}
\end{lem}

\begin{proof}
If $\gamma_{1} = \gamma_{2}$ we are done; suppose $\gamma_{1} < \gamma_{2}$. Define $x(a) := \gamma_{1}$ and $h(a) = \gamma_{2} - \gamma_{1} > 0$. By Taylor's theorem,
\begin{equation*}
    \left| u(\gamma_{1}) - u(\gamma_{2}) \right| = h(a) \cdot \left| u'(x(a)) + R_{1}(x(a)) \right| \leq \dfrac{h(a)}{(x(a)-s_{0})^{2}} + h(a) \cdot |R_{1}(x(a))|,
\end{equation*}
where we used assumption \ref{u:3} on $u$ and where $R_{1}$ is the first order Taylor remainder. But now
\begin{equation*}
    \dfrac{h(a)}{(x(a)-s_{0})^{2}} = \dfrac{\gamma_{2}-\gamma_{1}}{(\gamma_{1}-s_{0})^{2}} < \dfrac{r_{a}-\ell_{a}}{(\ell_{a}-s_{0})^{2}} < C_{1}
\end{equation*}
for some constant $C_{1}$ independent of $a$, where we used assumption \ref{G:l_i and r_i condition squared} on $G$. Since 
\begin{equation*}
    h(a) \cdot |R(x(a))| \to 0 \text{ as } a \to \infty,
\end{equation*}
$h(a) \cdot |R(x(a))| < C_{2}$ for some constant $C_{2}$. Taking $D = C_{1} + C_{2}$ completes the proof.
\end{proof}

We are now ready to prove Lemmas \ref{lem:5.19} and \ref{lem:5.20}.

\begin{proof}[Proof of Lemma~\ref{lem:5.19}]
We will first show that such an $m_{0}$ exists, and then give an algorithm to compute it.

Note that the sum in the expression for $\Phi(\omega)$ converges, because its tail converges:
\begin{equation*}
    \sum_{i=n+1}^{\infty} \left( \eta_{0}(\omega) \cdots \eta_{i-1}(\omega) \right)^{\nu} \cdot u(\eta_{i}(\omega)) \leq u(\eta_{n+1}(\omega)) \cdot (\eta_{n}(\omega))^{\nu} \cdot \sum_{i=n+1}^{\infty} (\eta_{n+1}(\omega))^{\nu} < \infty
\end{equation*}
since $(\eta_{n+1}(\omega))^{\nu} < s_{1} \leq 1$ for $\nu > 0$. Hence there is an $m_{1}>1$ such that the tail of the sum
\begin{equation*}
    \sum_{i \geq n+m_{1}} \left( \eta_{1}(\omega) \cdots \eta_{i-1}(\omega) \right)^{\nu} \cdot u(\eta_{i}(\omega)) < \dfrac{\varepsilon}{2}.
\end{equation*}
If needed, additionally make $m_{1}$ large enough that Lemma~\ref{lem:log_ratio_estimates} \ref{lem:u_diff_bounded:3} applies for all $m \geq m_{1}$. We will show how to choose $m_{0}>m_{1}$ to satisfy the conclusion of the lemma.

By Lemma~\ref{lem:log_ratio_estimates} \ref{lem:u_diff_bounded:2} and \ref{lem:u_diff_bounded:3}, for any $\beta^{I}$ and any $i \leq n+m_{1}$ we have

\begin{align*}
    &\hspace{12px} \left| \log \dfrac{ \left( \eta_{0}(\beta^{I}) \cdots \eta_{i-1}(\beta^{I}) \right)^{\nu} \cdot u(\eta_{i}(\beta^{I})) }{ \left( \eta_{0}(\omega) \cdots \eta_{i-1}(\omega) \right)^{\nu} \cdot u(\eta_{i}(\omega)) } \right| = \left| \nu \cdot \log \frac{ \eta_{0}(\beta^{I}) \cdots \eta_{i-1}(\beta^{I}) }{ \eta_{0}(\omega) \cdots \eta_{i-1}(\omega) } + \log \frac{ u(\eta_{i}(\beta^{I})) }{ u(\eta_{i}(\omega)) } \right| \\
    &< \nu \cdot \left( \sum_{j=1}^{i-1} m_{g} \cdot \sigma \cdot \tau^{1 + \frac{j+1-(n+m)}{\kappa}} \right) + f(m) < \nu \sigma \cdot m_{g} \cdot \tau^{1 + \frac{2-(n+m)}{\kappa}} \left( \dfrac{\tau^{i}-1}{\tau^{1/\kappa}-1} \right) + f(m) \\
    &< \nu \sigma \cdot m_{g} \cdot \tau^{1 + \frac{2-(n+m)}{\kappa}} \left( \dfrac{\tau^{n+m_{1}}}{\tau^{1/\kappa}-1} \right) + f(m) \leq \nu \sigma \cdot m_{g} \cdot \dfrac{\tau^{1+m_{1}+\frac{2-m}{\kappa}}}{\tau^{1/\kappa}-1} + f(m).
\end{align*}
We can choose $m_{0}$ sufficiently large so that $\exp(\nu \sigma \cdot m_{g} \cdot \dfrac{\tau^{1+m_{1}+\frac{2-m}{\kappa}}}{\tau^{1/\kappa}-1} + f(m)) > 1 - \dfrac{\varepsilon}{2 \cdot \Phi(\omega)}$ for $m \geq m_{0}$ and also
\begin{equation*}
    \left( \eta_{1}(\beta^{I}) \cdots \eta_{i-1}(\beta^{I}) \right)^{\nu} \cdot u(\eta_{i}(\beta^{I})) > \left( 1 - \dfrac{\varepsilon}{2 \cdot \Phi(\omega)} \right) \cdot \left( \eta_{1}(\omega) \cdots \eta_{i-1}(\omega) \right)^{\nu} \cdot u(\eta_{i}(\omega))
\end{equation*}
for $i \leq n+m_{1}$. Now, for any $\beta^{I}$ we have
\begin{align*}
    \Phi(\beta^{I}) &\geq \sum_{i=1}^{n+m_{1}-1} \left( \eta_{1}(\beta^{I}) \cdots \eta_{i-1}(\beta^{I}) \right)^{\nu} \cdot u(\eta_{i}(\beta^{I})) \\
&> \sum_{i=1}^{n+m_{1}-1} \left( 1 - \dfrac{\varepsilon}{2 \cdot \Phi(\omega)} \right) \cdot \left( \eta_{1}(\omega) \cdots \eta_{i-1}(\omega) \right)^{\nu} \cdot u(\eta_{i}(\omega))  \\
&> \left( 1 - \dfrac{\varepsilon}{2 \cdot \Phi(\omega)} \right) \left( \Phi(\omega) - \dfrac{\varepsilon}{2} \right) > \Phi(\omega) - \varepsilon.
\end{align*}
Since the symbolic representation of $\omega$ ends with all ones, for any specific $(a_{1}, a_{2}, \ldots, a_{n})$ we can compute $\Phi(\omega)$ by Proposition~\ref{prop:computes_Brjuno_value}. This allows us to compute $m_{1}$ by iteratively checking whether 
\begin{equation*}
    \sum_{i \geq n+m_{1}} \left( \eta_{0}(\omega) \cdots \eta_{i-1}(\omega) \right)^{\nu} \cdot u(\eta_{i}(\omega)) < \dfrac{\varepsilon}{2}
\end{equation*}
for each $m_{1}>1$. Then for each candidate $m_{0} > m_{1}$, since we can compute $\Phi(\omega)$ and $\left( \eta_{0}(\omega) \cdots \eta_{i-1}(\omega) \right)^{\nu} \cdot u(\eta_{i}(\omega))$ and left-compute $\left( \eta_{0}(\beta^{I}) \cdots \eta_{i-1}(\beta^{I}) \right)^{\nu} \cdot u(\eta_{i}(\beta^{I}))$ for each individual $i$ by assumption \ref{u:5} on $u$ and computability assumption \ref{G:computability} on $G$, we can iteratively check whether the following two conditions hold:
\begin{enumerate}[label={(\arabic*)}]
    \item \label{L2:1} $\exp \left(\nu \sigma \cdot m_{g} \cdot \dfrac{\tau^{1+m_{1}+\frac{2-m}{\kappa}}}{\tau^{1/\kappa}-1} + f(m) \right) > 1 - \dfrac{\varepsilon}{2 \cdot \Phi(\omega)}$
    \item \label{L2:2} For every $i \leq n+m_{1}$,
    \begin{equation*}
    \left( \eta_{0}(\beta^{I}) \cdots \eta_{i-1}(\beta^{I}) \right)^{\nu} \cdot u(\eta_{i}(\beta^{I})) > \left( 1 - \dfrac{\varepsilon}{2 \cdot \Phi(\omega)} \right) \cdot \left( \eta_{0}(\omega) \cdots \eta_{i-1}(\omega) \right)^{\nu} \cdot u(\eta_{i}(\omega)).
\end{equation*}
\end{enumerate}
We only need to be able to left-compute $\left( \eta_{0}(\beta^{I}) \cdots \eta_{i-1}(\beta^{I}) \right)^{\nu} \cdot u(\eta_{i}(\beta^{I}))$ because of the strict inequality in condition \ref{L2:2}.
Since we have shown there must exist $m_{1}>m_{0}$ for which these hold, this algorithm terminates and we eventually find the $m_{1}$ we are looking for.
\end{proof}

\begin{proof}[Proof of Lemma~\ref{lem:5.20}]
Let $i \leq k-1$ for some $k>1$. We will first show that 
\begin{equation*}
    \rho_{1} < \dfrac{u(\eta_{i}(\omega_{k}))}{u(\eta_{i}(\omega))} < \rho_{2}
\end{equation*}
for some $0 < \rho_{1} < \rho_{2} < \infty$, independent of $i$ and $k$. Since $i \leq k-1$, the symbolic representations of $\eta_{i}(\omega_{k}), \eta_{i}(\omega)$ coincide in the first two terms, call them $a$ and $b$. So, we can write
\begin{align*}
    &x := \eta_{i}(\omega_{k}) = [a, b, T_{1}], \\
    &y := \eta_{i}(\omega) = [a, b, T_{2}],
\end{align*}
where $T_{1}, T_{2}$ are the infinite tails of the symbolic expansions of $ \eta_{i}(\omega_{k}),  \eta_{i}(\omega)$ respectively. Since $u$ is $\mathcal{C}^{1}$ we have
\begin{equation*}
    u(y) = u(x) + \int_{x}^{y} u'(t) \,dt
\end{equation*}
and so by assumption \ref{u:3} on $u$ and \ref{G:l_i and r_i condition squared} on $G$,
\begin{align*}
    \left| \dfrac{u(y)}{u(x)} - 1 \right| &= \dfrac{1}{u(x)} \left| \int_{x}^{y} u'(t) \,dt \right| \leq \dfrac{1}{u(x)} \left| \int_{x}^{y} \dfrac{C}{(t-s_{0})^{2}} \,dt \right| = \dfrac{C}{u(x)} \left| \dfrac{1}{x-s_{0}} - \dfrac{1}{y-s_{0}} \right| \\[0.8em]
    &< \dfrac{C}{u(x)} \left( \dfrac{1}{\ell_{i}-s_{0}} - \dfrac{1}{r_{i}-s_{0}} \right) < \dfrac{C}{u(x)} \cdot \dfrac{r_{i} - \ell_{i}}{(\ell_{i}-s_{0})^{2}} < \dfrac{\tilde{C}}{u(x)}.
\end{align*}
for some constant $\tilde{C}>0$. By assumption \ref{u:1} on $u$, there is some $l_{1} \in (0, 1)$ such that $\dfrac{1}{u(\xi)} < \dfrac{1}{2\tilde{C}}$ for $\xi \in (0, l_{1})$. Thus if $a$ is large enough such that $x, y \in (0, l_{1})$, say $a \geq N_{1} \in \mathbb{Z}^{+}$, then
\begin{equation*}
    \left| \dfrac{u(x)}{u(y)} - 1 \right| < \dfrac{\tilde{C}}{2\tilde{C}} = \dfrac{1}{2}.
\end{equation*}
If $1 < a < N_{1}$, then $x, y \in [l_{1}', l_{2}']$ for some $0<l_{1}'<l_{2}'<1$, thus $u(x), u(y)$ are bounded away from $0$ and $\infty$, so $\rho_{1}' < \frac{u(x)}{u(y)} < \rho_{2}'$ for some $0 < \rho_{1}' < \rho_{2}' < \infty$.

Now suppose $a=1$. If $s_{1}<1$, then $x, y \in J_{1}$ with $J_{1}$ bounded away from $0$ and $1$, thus $u$ is bounded away from $0$ and infinity on $J_{1}$. Now suppose $s_{1}=1$. It follows immediately from assumption \ref{u:2} on $u$ that for large enough $b$, say $b \geq N_{2}$, $\frac{u(x)}{u(y)}$ is bounded away from $0$ and infinity. Therefore $\rho_{1}'' < \frac{u(x)}{u(y)} < \rho_{2}''$ for some $0<\rho_{1}''<\rho_{2}''<\infty$.

Putting all this together, letting $\rho_{1} = \min \{ \frac{1}{2}, \rho_{1}', \rho_{1}'' \} > 0$ and $\rho_{2} = \max \{ \frac{1}{2}, \rho_{2}', \rho_{2}'' \} < \infty$ we have for any $\omega_{k}$ and $\omega$ that $\rho_{1} < \frac{u(\omega_{k})}{u(\omega)} < \rho_{2}$ and so $\left| \log \frac{u(\omega_{k})}{u(\omega)} \right| < D_{1} < \infty$ for some constant $D_{1}>0$.

Returning to the proof of the Lemma, since the sum in the expression for $\Phi(\omega)$ converges we can split this sum as
\begin{equation*}
    \Phi(\omega) = \underbrace{\sum_{i=1}^{s} \left( \eta_{0}(\omega) \cdots \eta_{i-1}(\omega) \right)^{\nu} \cdot u(\eta_{i}(\omega))}_{\text{"head"}} + \underbrace{\sum_{i=s+1}^{\infty} \left( \eta_{0}(\omega) \cdots \eta_{i-1}(\omega) \right)^{\nu} \cdot u(\eta_{i}(\omega))}_{\text{"tail"}}
\end{equation*}
so that "tail" $ < \dfrac{\varepsilon}{2 \left( \exp \left( \nu \sigma \cdot m_{g} \cdot \frac{\tau}{\tau^{1/\kappa}-1} + D_{1} \right) \right)}$. If needed, additionally make $s$ large enough so that
\begin{equation*}
    \dfrac{\rho^{\nu s/2}}{1-\rho^{\nu/2}} < \dfrac{\varepsilon}{4 \left( \exp \left( \nu \sigma \cdot m_{g} \cdot \frac{\tau}{\tau^{1/\kappa}-1} \right) \cdot D_{2} + u(\varphi) \right)},
\end{equation*}
where $\varphi = [1, 1, \ldots]$, $\rho$ is as in Lemma~\ref{lem:consecutive_eta_bound}, and $D_{2}$ is the constant from Lemma~\ref{lem:u_diff_bounded}.

Since $\sum_{i=1}^{s} \left( \eta_{0}(\cdot) \cdots \eta_{i-1}(\cdot) \right)^{\nu} \cdot u(\eta_{i}(\cdot))$ is a finite sum of functions continuous on a neighbourhood of $\omega$, it is continuous on some interval containing $\omega$. Noting that $\omega_{k} \to \omega$, there is some $m>s$ such that for any $k \geq m$,
\begin{equation*}
    \sum_{i=1}^{s} \left( \eta_{0}(\omega_{k}) \cdots \eta_{i-1}(\omega_{k}) \right)^{\nu} \cdot u(\eta_{i}(\omega_{k})) < \sum_{i=1}^{s} \left( \eta_{0}(\omega) \cdots \eta_{i-1}(\omega) \right)^{\nu} \cdot u(\eta_{i}(\omega)) + \varepsilon/4.
\end{equation*}
We now want to bound the change from 
\begin{equation*}
    \left( \eta_{0}(\omega_{k}) \cdots \eta_{i-1}(\omega_{k}) \right)^{\nu} \cdot u(\eta_{i}(\omega_{k})) \text{\hspace{5px} to \hspace{5px}} \left( \eta_{0}(\omega) \cdots \eta_{i-1}(\omega) \right)^{\nu} \cdot u(\eta_{i}(\omega))
\end{equation*}
for $i>s$. We consider these terms individually for $s<i \leq k-1, i=k$, and $i \geq k+1$.

\begin{itemize}
    \item For $s<i \leq k-1$. Taking $n+m := k+1$ in Lemma~\ref{lem:log_ratio_estimates} \ref{lem:u_diff_bounded:2}, since $\omega_{k}$ and $\omega$ coincide in the first $k=n+m-1$ terms, we have
\begin{align*}
    &\left| \log \frac{ \left( \eta_{0}(\omega_{k}) \cdots \eta_{i-1}(\omega_{k}) \right)^{\nu} \cdot u(\eta_{i}(\omega_{k})) }{ \left( \eta_{0}(\omega) \cdots \eta_{i-1}(\omega) \right)^{\nu} \cdot u(\eta_{i}(\omega) )} \right| = \left| \nu \cdot \log \frac{\eta_{0}(\omega_{k}) \cdots \eta_{i-1}(\omega_{k}) }{\eta_{0}(\omega) \cdots \eta_{i-1}(\omega) } + \log \frac{u(\eta_{i}(\omega_{k}))}{u(\eta_{i}(\omega))} \right| \\
    &< \nu \cdot \left( \sum_{j=1}^{i-1} m_{g} \cdot \sigma \cdot \tau^{1 + \frac{j-k}{\kappa}} \right) + D_{1} < \nu \sigma \cdot m_{g} \cdot \dfrac{\tau^{1+\frac{i-k}{\kappa}}}{\tau^{1/\kappa}-1} + D_{1} \leq \nu \sigma \cdot m_{g} \cdot \dfrac{\tau^{1-\frac{1}{\kappa}}}{\tau^{1/\kappa}-1} + D_{1}.
\end{align*}
Hence in this case each term can increase by a factor of at most $\exp \left(\nu \sigma \cdot m_{g} \cdot \frac{\tau^{1-\frac{1}{\kappa}}}{\tau^{1/\kappa}-1} + D_{1} \right)$.

    \item For $i=k$. We have
    \begin{equation*}
        \left| u(\eta_{i}(\omega)) - u(\eta_{i}(\omega_{k})\right)| = \left| u\left( [a_{k}, \ldots] \right) - u\left( [a_{k}, \ldots] \right) \right| < D_{2},
    \end{equation*}
    where we had previously gotten $D_{2}$ from Lemma~\ref{lem:u_diff_bounded}. Now by Lemma~\ref{lem:log_ratio_estimates} \ref{lem:u_diff_bounded:2},
\begin{align*}
    \left| \log \frac{ \left( \eta_{0}(\omega_{k}) \cdots \eta_{i-1}(\omega_{k}) \right)^{\nu}}{ \left( \eta_{0}(\omega) \cdots \eta_{i-1}(\omega) \right)^{\nu}} \right|
    &= \nu \cdot \left| \log \frac{\eta_{0}(\omega_{k}) \cdots \eta_{i-1}(\omega_{k})}{\eta_{0}(\omega) \cdots \eta_{i-1}(\omega)} \right| < \nu \cdot \left( \sum_{j=1}^{i-1} m_{g} \cdot \sigma \cdot \tau^{1 + \frac{j-k}{\kappa}} \right) \\
    &< \nu \sigma \cdot m_{g} \cdot \dfrac{\tau^{1+\frac{i-k}{\kappa}}}{\tau^{1/\kappa}-1} = \nu \sigma \cdot m_{g} \cdot \dfrac{\tau}{\tau^{1/\kappa}-1},
\end{align*}
    so $\left( \eta_{0}(\omega_{k}) \cdots \eta_{i-1}(\omega_{k}) \right)^{\nu} < \exp \left( \nu \sigma \cdot m_{g} \cdot \frac{\tau}{\tau^{1/\kappa}-1} \right) \cdot \left( \eta_{0}(\omega) \cdots \eta_{i-1}(\omega) \right)^{\nu}$. Combining these estimates and using the value $\rho<1$ from Lemma~\ref{lem:consecutive_eta_bound}, we get
    \begin{align*}
        &\left( \eta_{0}(\omega_{k}) \cdots \eta_{i-1}(\omega_{k}) \right)^{\nu} \cdot u(\eta_{i}(\omega_{k})) \leq \exp \left( \nu \sigma \cdot m_{g} \cdot \frac{\tau}{\tau^{1/\kappa}-1} \right) \cdot \left( \eta_{0}(\omega) \cdots \eta_{i-1}(\omega) \right)^{\nu} \cdot \left( u(\eta_{i}(\omega)) + D_{2} \right) \\
        &< \exp \left( \nu \sigma \cdot m_{g} \cdot \frac{\tau}{\tau^{1/\kappa}-1} \right) \cdot \left( \eta_{0}(\omega) \cdots \eta_{i-1}(\omega) \right)^{\nu} \cdot u(\eta_{i}(\omega)) + \exp \left( \nu \sigma \cdot m_{g} \cdot \frac{\tau}{\tau^{1/\kappa}-1} \right) \cdot D_{2} \cdot \rho^{\nu (i-1)/2}.
    \end{align*}
    \item For $i \geq k+1$. Noting that $u(\eta_{i}(\omega_{k})) = u(\eta_{i}(\omega_{k}))$, by Lemma~\ref{lem:consecutive_eta_bound} we have
    \begin{equation*}
        \left( \eta_{0}(\omega_{k}) \cdots \eta_{i-1}(\omega_{k}) \right)^{\nu} \cdot u(\eta_{i}(\omega_{k})) < \rho^{\nu (i-1)/2} \cdot u(\varphi).
    \end{equation*}
\end{itemize}
Therefore for any $i>s$,
\begin{align*}
    &\left( \eta_{0}(\omega_{k}) \cdots \eta_{i-1}(\omega_{k}) \right)^{\nu} \cdot u(\eta_{i}(\omega_{k})) < \exp \left( \nu \sigma \cdot m_{g} \cdot \frac{\tau}{\tau^{1/\kappa}-1} + D_{1} \right) \cdot \left( \eta_{0}(\omega) \cdots \eta_{i-1}(\omega) \right)^{\nu} \cdot u(\eta_{i}(\omega)) \\
    &\hspace{180px} + \left( \exp \left( \nu \sigma \cdot m_{g} \cdot \frac{\tau}{\tau^{1/\kappa}-1} \right) \cdot D_{2} + u(\varphi) \right) \cdot \rho^{\nu (i-1)/2}.
\end{align*}
Finally, noting that
\begin{equation*}
    \sum_{i=s+1}^{\infty} \rho^{\nu (i-1)/2} = \dfrac{\rho^{\nu s/2}}{1-\rho^{\nu/2}},
\end{equation*}
for such a $k$ we have
\begin{align*}
    \Phi(\omega_{k}) &= \text{"head"}(\omega_{k}) + \text{"tail"}(\omega_{k}) \\
    &< \left( \text{"head"}(\omega) + \dfrac{\varepsilon}{4} \right) + \left(  \exp \left( \nu \sigma \cdot m_{g} \cdot \frac{\tau}{\tau^{1/\kappa}-1} + D_{1} \right) \cdot \text{"tail"}(\omega) \right.\ \\
    & \hspace{105px}+ \left.\ \left( \exp \left( \nu \sigma \cdot m_{g} \cdot \frac{\tau}{\tau^{1/\kappa}-1} \right) \cdot D_{2} + u(\varphi) \right) \cdot \dfrac{\rho^{\nu s/2}}{1-\rho^{\nu/2}} \right) \\
    &< \text{"head"}(\omega) + \dfrac{\varepsilon}{4} + \exp \left( \nu \sigma \cdot m_{g} \cdot \frac{\tau}{\tau^{1/\kappa}-1} + D_{1} \right) \cdot \dfrac{\varepsilon}{2 \left( \exp \left( \nu \sigma \cdot m_{g} \cdot \frac{\tau}{\tau^{1/\kappa}-1} + D_{1} \right) \right)} \\
    & \hspace{20px}+ \left( \exp \left( \nu \sigma \cdot m_{g} \cdot \frac{\tau}{\tau^{1/\kappa}-1} \right) \cdot D_{2} + u(\varphi) \right) \cdot \dfrac{\varepsilon}{4 \left( \exp \left( \nu \sigma \cdot m_{g} \cdot \frac{\tau}{\tau^{1/\kappa}-1} \right) \cdot D_{2} + u(\varphi) \right)} \\
    &= \text{"head"}(\omega) + \dfrac{\varepsilon}{4} + \dfrac{\varepsilon}{2} + \dfrac{\varepsilon}{4} < \Phi(\omega) + \varepsilon.
\end{align*}
\end{proof}

\subsection{Proof of Lemma~\ref{lem:5.18}}

Write
\begin{equation*}
    \Phi^{-}(\omega) = \Phi(\omega) - s(n+m) \cdot \left( \eta_{1}(\omega) \eta_{2}(\omega) \cdots \eta_{n+m-1}(\omega) \right)^{\nu} \cdot u(\eta_{n+m}(\omega)).
\end{equation*}
The value of the integer $m>0$ is yet to be determined.

\begin{lem}\label{lem:Phi_minus_N_1_diff}
For any $\omega$ of the form as in Lemma~\ref{lem:5.18} and for any $\varepsilon>0$, there is an $m_{0}>0$ such that, for any $N$ and any $m \geq m_{0}$,
\begin{equation*}
    | \Phi^{-}(\beta^{N}) - \Phi^{-}(\beta^{1}) | < \frac{\varepsilon}{4}.
\end{equation*}
\end{lem}

\begin{proof}
Noting that $\eta_{n+m-1}(\beta^{N}) \in (\ell_{1}, s_{1})$ for all $N \in \mathbb{Z}^{+}$, by Lemma~\ref{lem:u_bounded} we have some $C>0$ such that $0 \leq u(\eta_{n+m-1}(\beta^{N})) < C$ for all $N$. Now, the sum in the expression for $\Phi(\beta^{1})$ converges because its tail converges absolutely:
\begin{equation*}
    \sum_{i=n+1}^{\infty} \left| s(i) \cdot \left( \eta_{0}(\beta^{1}) \cdots \eta_{i-1}(\beta^{1}) \right)^{\nu} \cdot u(\eta_{i}(\beta^{1})) \right| = u(\beta^{1}) \cdot \sum_{i=n+1}^{\infty} \left( (\beta^{1})^{\nu} \right)^{i} < \infty
\end{equation*}
since $(\beta^{1})^{\nu} < s_{1} \leq 1$ for $\nu > 0$. Hence there is an $m_{1}>1$, make it large enough that Lemma~\ref{lem:log_ratio_estimates} \ref{lem:u_diff_bounded:3} applies for all $m \geq m_{1}$, such that the tail of the sum
\begin{equation*}
    \sum_{i\geq n+m_{1}} s(i) \cdot \left( \eta_{0}(\beta^{1}) \cdots \eta_{i-1}(\beta^{1}) \right)^{\nu} \cdot u(\eta_{i}(\beta^{1})) < \frac{\varepsilon}{4(1 + D)},
\end{equation*}
where $M>0$ is as in Lemma~\ref{lem:log_ratio_estimates} \ref{lem:u_diff_bounded:3}, and $D$ is defined as
\begin{equation*}
    D = \exp \left( \nu \left( \frac{m_{g} \sigma \tau^{3}}{\tau^{1/\kappa} -1} + \left| \log \left( \frac{C}{u(\ell_{1})} \right) \right| \right) + M \right).
\end{equation*}
We will now show how to choose $m_{0}>m_{1}$ to satisfy the conclusion of the Lemma.

We bound the influence of the change from $\beta^{1}$ to $\beta^{N}$ using Lemma~\ref{lem:log_ratio_estimates} \ref{lem:u_diff_bounded:2} and \ref{lem:u_diff_bounded:3}. The influence on each of the "head elements" ($i<n+m_{1}$) is bounded by
\begin{align*}
    &\left| \log \frac{ s(i) \cdot \left( \eta_{0}(\beta^{1}) \cdots \eta_{i-1}(\beta^{1}) \right)^{\nu} \cdot u(\eta_{i}(\beta^{1})) }{ s(i) \cdot \left( \eta_{0}(\beta^{N}) \cdots \eta_{i-1}(\beta^{N}) \right)^{\nu} \cdot u(\eta_{i}(\beta^{N})) } \right| = \left| \nu \cdot \log \frac{ \eta_{0}(\beta^{1}) \cdots \eta_{i-1}(\beta^{1}) }{ \eta_{0}(\beta^{N}) \cdots \eta_{i-1}(\beta^{N}) } + \log \frac{ u(\eta_{i}(\beta^{1})) }{ u(\eta_{i}(\beta^{N})) } \right| \\
&< \nu \cdot \left( \sum_{j=1}^{i-1} m_{g} \cdot \sigma \cdot \tau^{1 + \frac{j+1-(n+m)}{\kappa}} \right) + f(m) \leq \nu \sigma \cdot m_{g} \cdot \dfrac{\tau^{m_{1} + \frac{2 - m}{\kappa}}}{\tau^{1/\kappa}-1} + f(m).
\end{align*}

By making $m$ sufficiently large (i.e., by choosing a sufficiently large $m_{0}$) we can ensure that
\begin{equation*}
    1 - \frac{\varepsilon}{4(1 + D) \cdot \Phi(\beta^{1})} < \frac{ s(i) \cdot \left( \eta_{0}(\beta^{N}) \cdots \eta_{i-1}(\beta^{N}) \right)^{\nu} \cdot u(\eta_{i}(\beta^{N})) }{ s(i) \cdot \left( \eta_{0}(\beta^{1}) \cdots \eta_{i-1}(\beta^{1}) \right)^{\nu} \cdot u(\eta_{i}(\beta^{1})) } < 1 + \frac{\varepsilon}{4(1 + D) \cdot \Phi(\beta^{1})}.
\end{equation*}
Hence
\begin{align*}
    &\left| s(i) \cdot \left( \eta_{0}(\beta^{N}) \cdots \eta_{i-1}(\beta^{N}) \right)^{\nu} \cdot u(\eta_{i}(\beta^{N})) - s(i) \cdot \left( \eta_{0}(\beta^{1}) \cdots \eta_{i-1}(\beta^{1}) \right)^{\nu} \cdot u(\eta_{i}(\beta^{1})) \right| \\
    &< \frac{\varepsilon}{4(1 + D) \cdot \Phi(\beta^{1})} \cdot \left( \eta_{0}(\beta^{1}) \cdots \eta_{i-1}(\beta^{1}) \right)^{\nu} \cdot u(\eta_{i}(\beta^{1})).
\end{align*}
Adding the inequalities for $i=1, 2, \ldots, n+m_{1}-1$, we obtain
\begin{align*}
    & \bigg\rvert \sum_{i=1}^{n+m_{1}-1} s(i) \cdot \left( \eta_{0}(\beta^{N}) \cdots \eta_{i-1}(\beta^{N}) \right)^{\nu} \cdot u(\eta_{i}(\beta^{N})) - \sum_{i=1}^{n+m_{1}-1} s(i) \cdot \left( \eta_{0}(\beta^{1}) \cdots \eta_{i-1}(\beta^{1}) \right)^{\nu} \cdot u(\eta_{i}(\beta^{1})) \bigg\rvert \\
    &< \frac{\varepsilon}{4(1 + D) \cdot \Phi(\beta^{1})} \sum_{i=1}^{n+m_{1}-1} \left( \eta_{0}(\beta^{1}) \cdots \eta_{i-1}(\beta^{1}) \right)^{\nu} \cdot u(\eta_{i}(\beta^{1})) = \frac{\varepsilon}{4(1 + D)}
\end{align*}
Thus the influence on the "head" of $\Phi^{-}$ is bounded by $\dfrac{\varepsilon}{4(1 + D)}$. To bound the influence on the "tail", we consider three kinds of terms $s(i) \cdot \left( \eta_{0}(\beta^{N}) \cdots \eta_{i-1}(\beta^{N}) \right)^{\nu} \cdot u(\eta_{i}(\beta^{N}))$: those for which $n+m_{1} \leq i \leq n+m-2$, $i = m+n-1$, and $i \geq m+n+1$ (recall that $i=n+m$ is not in $\Phi^{-}$).

\begin{itemize}
    \item For $n+m_{1} \leq i \leq n+m-2$. By Lemma~\ref{lem:log_ratio_estimates} \ref{lem:u_diff_bounded:2} and \ref{lem:u_diff_bounded:3}, and by the work above,
\begin{align*}
    &\left| \log \frac{ s(i) \cdot \left( \eta_{0}(\beta^{1}) \cdots \eta_{i-1}(\beta^{1}) \right)^{\nu} \cdot u(\eta_{i}(\beta^{1})) }{ s(i) \cdot \left( \eta_{0}(\beta^{N}) \cdots \eta_{i-1}(\beta^{N}) \right)^{\nu} \cdot u(\eta_{i}(\beta^{N})) } \right| < \nu \cdot \left( \sum_{j=1}^{i-1} m_{g} \cdot \sigma \cdot \tau^{1 + \frac{j+1-(n+m)}{\kappa}} \right) + f(m) \\
&\leq \nu \sigma \cdot m_{g} \cdot \tau^{1+\frac{2-(n+m)}{\kappa}} \left( \dfrac{\tau^{i}}{\tau^{1/\kappa}-1} \right) + f(m) \leq \nu \sigma \cdot m_{g} \cdot \dfrac{\tau}{\tau^{1/\kappa}-1} + M
\end{align*}
Where $M>0$ is as in Lemma~\ref{lem:log_ratio_estimates} \ref{lem:u_diff_bounded:3}. Hence in this case each term can increase or decrease by a factor of at most $\exp \left( \frac{\nu \sigma m_{g} \tau}{\tau^{1/ \kappa} - 1} + M \right)$.

    \item For $i=n+m-1$. Recall that for all $N \in \mathbb{Z}^{+}$, $0 \leq u(\eta_{n+m-1}(\beta^{N})) < C$. Thus we have
\begin{align*}
    &\log \frac{ s(i) \cdot \left( \eta_{0}(\beta^{N}) \cdots \eta_{i-1}(\beta^{N}) \right)^{\nu} \cdot u(\eta_{i}(\beta^{N})) }{ s(i) \cdot \left( \eta_{0}(\beta^{1}) \cdots \eta_{i-1}(\beta^{1}) \right)^{\nu} \cdot u(\eta_{i}(\beta^{1})) } \leq \log \frac{ \left( \eta_{0}(\beta^{N}) \cdots \eta_{i-1}(\beta^{N}) \right)^{\nu} \cdot C }{ \left( \eta_{0}(\beta^{1}) \cdots \eta_{i-1}(\beta^{1}) \right)^{\nu} \cdot u(\ell_{1}) } \\
&< \nu \cdot \left( \sum_{j=1}^{n+m-2} m_{g} \cdot \sigma \cdot \tau^{1 + \frac{j+1-(n+m)}{\kappa}} + \log \left( \dfrac{C}{u(\ell_{1})} \right) \right) < \nu \cdot \left( \dfrac{m_{g} \cdot \sigma \cdot \tau^{2}}{\tau^{1/ \kappa} -1} + \log \left( \dfrac{C}{u(\ell_{1})} \right) \right)
\end{align*}
Hence this term could increase or decrease by a factor of at most $\exp \left( \nu \cdot \left( \frac{m_{g} \sigma \tau^{2}}{\tau^{1/ \kappa} -1} + \log \left( \frac{C}{u(\ell_{1})} \right) \right) \right)$.

    \item For $i \geq n+m+1$. Note that the $\eta_{j}$ for $j>n+m$ are not affected by the change, and the change decreases $\eta_{n+m}$, so that $\eta_{n+m}(\beta^{N}) \leq \eta_{n+m}(\beta^{1})$ and thus $\left( \eta_{n+m}(\beta^{N}) \right)^{\nu} \leq \left( \eta_{n+m}(\beta^{1}) \right)^{\nu}$. Hence
\begin{align*}
    &\log \frac{ s(i) \cdot \left( \eta_{0}(\beta^{N}) \cdots \eta_{i-1}(\beta^{N}) \right)^{\nu} u(\eta_{i}(\beta^{N})) }{ s(i) \cdot \left( \eta_{0}(\beta^{1}) \cdots \eta_{i-1}(\beta^{1}) \right)^{\nu} u(\eta_{i}(\beta^{1})) }
= \log \frac{ \left( \eta_{0}(\beta^{N}) \cdots \eta_{n+m}(\beta^{N}) \right)^{\nu} }{ \left( \eta_{0}(\beta^{1}) \cdots \eta_{n+m}(\beta^{1}) \right)^{\nu} } \\
    &\leq \log \frac{ \left( \eta_{0}(\beta^{N}) \cdots \eta_{n+m-1}(\beta^{N}) \right)^{\nu} }{ \left( \eta_{0}(\beta^{1}) \cdots \eta_{n+m-1}(\beta^{1}) \right)^{\nu} } < \nu \cdot \sum_{j=1}^{n+m-1} m_{g} \cdot \sigma \cdot \tau^{1 + \frac{j+1-(n+m)}{\kappa}} < \dfrac{\nu \cdot m_{g} \cdot \sigma \cdot \tau^{3}}{\tau^{1/ \kappa} - 1}.
\end{align*}
So, in this case each term could increase or decrease by a factor of at most $\exp \left( \frac{\nu m_{g} \sigma \tau^{3}}{\tau^{1/ \kappa}-1} \right)$.
\end{itemize}

We see that after the change, each term of the tail could increase or decrease by a factor of
\begin{equation*}
    \exp \left( \nu \left( \frac{m_{g} \sigma \tau^{3}}{\tau^{1/\kappa} -1} + \left| \log \left( \frac{C}{u(\ell_{1})} \right) \right| \right) + M \right) = D
\end{equation*}
at most. So the value of the tail remains in the interval $\left[-\dfrac{D \cdot \varepsilon}{4(1 + D)}, \dfrac{D \cdot \varepsilon}{4(1 + D)} \right]$, hence the change in the tail is bounded by $\dfrac{D \cdot \varepsilon}{4(1 + D)}$.

Therefore, the total change in $\Phi^{-}$ is bounded by
\begin{equation*}
    \text{change in the "head" } + \text{ change in the "tail"} < \frac{\varepsilon}{4(1 + D)} + \frac{D \cdot \varepsilon}{4(1 + D)} = \frac{\varepsilon}{4}.
\end{equation*}
\end{proof}

\begin{lem}\label{lem:Phi_minus_N_N+1_diff}
For any $\varepsilon$ and for the same $m_{0}(\varepsilon)$ as in Lemma~\ref{lem:Phi_minus_N_1_diff}, for any $m \geq m_{0}$ and $N$,
\begin{equation*}
    |\Phi^{-}(\beta^{N}) - \Phi(\beta^{N+1})| < \frac{\varepsilon}{2}.
\end{equation*}
\end{lem}

\begin{proof}
We have
\begin{equation*}
    |\Phi^{-}(\beta^{N}) - \Phi(\beta^{N+1})| \leq |\Phi^{-}(\beta^{N}) - \Phi(\beta^{1})| + |\Phi^{-}(\beta^{N+1}) - \Phi(\beta^{1})| < \frac{\varepsilon}{4} + \frac{\varepsilon}{4} = \frac{\varepsilon}{2}.
\end{equation*}
\end{proof}

We will now take a closer look at the term
\begin{equation*}
    \Phi^{1}(\omega) := s(n+m) \cdot \left( \eta_{1}(\omega) \cdots \eta_{n+m-1}(\omega) \right)^{\nu} \cdot u(\eta_{n+m}(\omega)) = \Phi(\omega) - \Phi^{-}(\omega).
\end{equation*}

\begin{lem}\label{lem:Phi_1_diff}
There exists some $m$, which can be made arbitrarily large, such that for any $N$,
\begin{equation*}
    \Phi^{1}(\beta^{N+1}) - \Phi^{1}(\beta^{N}) < \frac{\varepsilon}{2}.
\end{equation*}
\end{lem}

\begin{proof}
We assume throughout that $s(n+m) = 1$. By assumption on $s(\cdot)$, there are infinitely many $m$ such that this holds, so we can take $m$ arbitrarily large within the course of the proof.

According to Lemma~\ref{lem:log_ratio_estimates} \ref{lem:u_diff_bounded:1},
\begin{align*}
    &\left| \log \frac{ s(n+m) \cdot \left( \eta_{0}(\beta^{N+1}) \cdots \eta_{n+m-1}(\beta^{N+1})) \right)^{\nu} }{ s(n+m) \cdot \left( \eta_{0}(\beta^{N}) \cdots \eta_{n+m-1}(\beta^{N})) \right)^{\nu} } \right| < \nu \cdot \sum_{i=1}^{n+m-1} \dfrac{g(N)}{\ell_{1}} \cdot \tau^{1 + \frac{i-(n+m)}{\kappa}} \\
    &= \nu \cdot g(N) \cdot \tau^{1 + \frac{1-(n+m)}{\kappa}} \cdot \dfrac{\tau^{n+m}-1}{\tau^{1/\kappa}-1} < \nu \cdot g(N) \cdot \dfrac{\tau^{2}}{\tau^{1/\kappa}-1}.
\end{align*}

Hence 
\begin{equation*}
    s(n+m) \cdot \left( \eta_{0}(\beta^{N+1} \cdots \eta_{n+m-1}(\beta^{N+1})) \right)^{\nu} < s(n+m) \cdot \left( \eta_{0}(\beta^{N} \cdots \eta_{n+m-1}(\beta^{N})) \right)^{\nu} \cdot \exp \left( \nu \cdot g(N) \cdot \dfrac{\tau^{2}}{\tau^{1/\kappa}-1} \right)
\end{equation*}
and
\begin{equation*}
    \Phi^{1}(\beta^{N+1}) < \Phi^{1}(\beta^{N}) \cdot \exp \left( \nu \cdot g(N) \cdot \dfrac{\tau^{2}}{\tau^{1/\kappa}-1} \right) \cdot \dfrac{ u(\eta_{n+m}(\beta^{N+1})) }{ u(\eta_{n+m}(\beta^{N})) }.
\end{equation*}
We require an auxiliary inequality. It is straightforward to show that for any $y \in (0, 1]$ and any $d>e^{c}/c$ we have $e^{c y} < 1 + d c y$. Taking $y = g(N)/(m_{g} \cdot \frac{\tau^{2}}{\tau^{1/\kappa}-1})$ and $c = \nu \cdot m_{g} \cdot \frac{\tau^{2}}{\tau^{1/\kappa}-1}$, for any $d>e^{\nu \cdot m_{g} \cdot \frac{\tau^{2}}{\tau^{1/\kappa}-1}}/(\nu \cdot m_{g} \cdot \frac{\tau^{2}}{\tau^{1/\kappa}-1})$ we have $e^{\nu \cdot g(N) \cdot \frac{\tau^{2}}{\tau^{1/\kappa}-1}} < 1 + d \nu \cdot g(N) \cdot \frac{\tau^{2}}{\tau^{1/\kappa}-1}$. Hence for any such $d$,
\begin{align*}
    \Phi^{1}(\beta^{N+1}) - \Phi^{1}(\beta^{N}) &< \Phi^{1}(\beta^{N}) \left( e^{\nu \cdot g(N) \cdot \frac{\tau^{2}}{\tau^{1/\kappa}-1}} \cdot \dfrac{ u(\eta_{n+m}(\beta^{N+1})) }{ u(\eta_{n+m}(\beta^{N})) } - 1 \right) \\
    &< \Phi^{1}(\beta^{N}) \left( \left( 1 + d \nu \cdot g(N) \cdot \frac{\tau^{2}}{\tau^{1/\kappa}-1} \right) \cdot \dfrac{ u(\eta_{n+m}(\beta^{N+1})) }{ u(\eta_{n+m}(\beta^{N})) } - 1 \right).
\end{align*}

For $\rho \in (0, 1)$ as in Lemma~\ref{lem:consecutive_eta_bound}, we have
\begin{equation*}
    \Phi^{1}(\beta^{N}) < \rho^{\nu \cdot (n+m-2)/2} \cdot u\left( \eta_{n+m}(\beta^{N}) \right).
\end{equation*}
Thus
\begin{align*}
    &\Phi^{1}(\beta^{N+1}) - \Phi^{1}(\beta^{N}) < \rho^{\nu \cdot (n+m-2)/2} \cdot u\left( \eta_{n+m}(\beta^{N}) \right) \cdot \left( \left( 1 + d \nu \cdot g(N) \cdot \frac{\tau^{2}}{\tau^{1/\kappa}-1} \right) \cdot \dfrac{ u(\eta_{n+m}(\beta^{N+1})) }{ u(\eta_{n+m}(\beta^{N})) } - 1 \right) \\[0.2em]
    &= \rho^{\nu \cdot (n+m-2)/2} \cdot \left( \left[ u\left( \eta_{n+m}(\beta^{N+1}) \right) - u\left( \eta_{n+m}(\beta^{N}) \right) \right] + d \nu \cdot g(N) \cdot \frac{\tau^{2}}{\tau^{1/\kappa}-1} \cdot u\left( \eta_{n+m}(\beta^{N+1}) \right) \right)
\end{align*}
For the remainder of the proof, we will show that this product can be made less than $\frac{\varepsilon}{2}$ by choosing $m$ large enough. Since the first term $\to 0$ as $m \to \infty$, it is enough to show that the product of the last two terms is bounded from above by a constant.

By assumption on $u$, we have $-u'(x) < \dfrac{C_{1}}{(x-s_{0})^{2}}$ for some constant $C_{1}>0$. Integrating both sides gives $u(x) < \dfrac{C_{2}}{x-s_{0}}$ for another constant $C_{2}>0$. Letting $x(N) := \eta_{n+m}(\beta^{N+1})$ and $h(N) := \eta_{n+m}(\beta^{N})-\eta_{n+m}(\beta^{N+1})>0$, we have
\begin{align*}
    &\dfrac{1}{h(N)} \cdot \left( \left[ u\left( \eta_{n+m}(\beta^{N+1}) \right) - u\left( \eta_{n+m}(\beta^{N}) \right) \right] + d \nu \cdot g(N) \cdot \frac{\tau^{2}}{\tau^{1/\kappa}-1} \cdot u\left( \eta_{n+m}(\beta^{N+1}) \right) \right) \\
&= - \left[ u'(x(N)) + R_{1}(x(N)) \right] + \dfrac{x(N) \cdot g(N) \cdot \frac{\tau^{2}}{\tau^{1/\kappa}-1}}{h(N)} \cdot \dfrac{d \nu \cdot u(x(N))}{x(N)},
\end{align*}

where $R_{1}(\cdot)$ is the Taylor remainder term with $\lim_{N \to \infty} R_{1}(x(N)) = 0$. 

We now want to find a lower bound on $h(N)$. Recalling assumption \ref{G:l_i and r_i condition} on $G$, we have
\begin{equation*}
    h(N) = [N, 1, 1, \ldots] - [N+1, 1, 1, \ldots] = G_{N}^{-1}(\varphi) - G_{N+1}^{-1}(\varphi) = \delta_{G}(N) > 0.
\end{equation*}
Taking $D>0$ as in assumption \ref{G:l_i and r_i condition} and recalling the definition of $g$ from \ref{G:function g}, we have
\begin{equation*}
    \dfrac{x(N) \cdot g(N) \cdot \frac{\tau^{2}}{\tau^{1/\kappa}-1}}{h(N)} < \dfrac{\eta_{n+m}(\beta^{N+1})}{\delta_{G}(N)} \cdot \dfrac{r_{N}-\ell_{N+1}}{\ell_{N+1}} \cdot \frac{\tau^{2}}{\tau^{1/\kappa}-1} < D  \cdot \frac{\tau^{2}}{\tau^{1/\kappa}-1}.
\end{equation*}
Since $\lim_{N \to \infty} R_{1}(x(N)) = 0$, we have $|R_{1}(x(N))| < C_{3}$ for some $C_{3}>0$. So,
\begin{align*}
    &-\left[ u'(x(N)) + R_{1}(x(N)) \right] + \dfrac{x(N) \cdot g(N) \cdot \frac{\tau^{2}}{\tau^{1/\kappa}-1}}{h(N)} \cdot \dfrac{d \nu \cdot u(x(N))}{x(N)} \\
    &< \dfrac{C_{1}}{(x(N))^{2}} + C_{3} + \dfrac{D \cdot \frac{\tau^{2}}{\tau^{1/\kappa}-1} \cdot d \nu}{x(N)} \cdot \dfrac{C_{2}}{(x(N))} < \dfrac{C_{4}}{(x(N))^{2}} < \dfrac{C}{h(N)}
\end{align*}
for some $C_{4}, C>0$, where the last inequality holds true because
\begin{equation*}
    h(N) < r_{N}-\ell_{N+1} < \ell_{N+1} \cdot g(N) < \eta_{n+m}(\beta^{N+1}) \cdot g(N) = x(N) \cdot g(N) < (x(N))^{2} \cdot m_{g}.
\end{equation*}
Therefore we have shown that for all $N$,
\begin{equation*}
    \left[ u\left( \eta_{n+m}(\beta^{N+1}) \right) - u\left( \eta_{n+m}(\beta^{N}) \right) \right] + d \nu \cdot g(N) \cdot \frac{\tau^{2}}{\tau^{1/\kappa}-1} \cdot u\left( \eta_{n+m}(\beta^{N+1}) \right) < C
\end{equation*}
for some constant $C>0$, finishing the proof of the Lemma.
\end{proof}

Lemmas \ref{lem:Phi_minus_N_N+1_diff} and \ref{lem:Phi_1_diff} yield the following.

\begin{lem}\label{Phi_diff}
There exists some $m$, which can be made arbitrarily large, such that for any $N$,
\begin{equation*}
    \Phi(\beta^{N+1}) - \Phi(\beta^{N}) < \varepsilon.
\end{equation*}
\end{lem}

\begin{proof}
Using Lemmas \ref{lem:Phi_minus_N_N+1_diff} and \ref{lem:Phi_1_diff}, for sufficiently large $m$ with $s(n+m)=1$ we have
\begin{equation*}
    \Phi(\beta^{N+1}) - \Phi(\beta^{N}) \leq \Phi^{-}(\beta^{N+1}) - \Phi^{-}(\beta^{N}) + \Phi^{1}(\beta^{N+1}) - \Phi^{1}(\beta^{N}) < \frac{\varepsilon}{2} + \frac{\varepsilon}{2} = \varepsilon.
\end{equation*}
\end{proof}

To complete the proof of Lemma~\ref{lem:5.18}, we will need the following statement.

\begin{lem}\label{lem:Phi_limit_unbounded}
For any $m$ satisfying $s(n+m) = 1$, we have
\begin{equation*}
    \lim_{N \to \infty} \Phi(\beta^{N}) = \infty.
\end{equation*}
\end{lem}

\begin{proof}
We first prove that $\lim_{N \to \infty} \Phi^{1}(\beta^{N}) = \infty$. By Lemma~\ref{lem:log_ratio_estimates} \ref{lem:u_diff_bounded:2},
\begin{equation*}
    \left| \log \frac{ s(n+m) \cdot \left( \eta_{0}(\beta^{N}) \cdots \eta_{n+m-1}(\beta^{N}) \right)^{\nu} }{ s(n+m) \cdot  \left( \eta_{0}(\beta^{1}) \cdots \eta_{n+m-1}(\beta^{1}) \right)^{\nu} } \right| < \nu \cdot \sum_{j=1}^{n+m-1} m_{g} \cdot \sigma \cdot \tau^{1 + \frac{j+1-(n+m)}{\kappa}} < \dfrac{\nu \cdot m_{g} \cdot \sigma \cdot \tau^{3}}{\tau^{1/ \kappa} - 1}.
\end{equation*}
Hence
\begin{equation*}
    s(n+m) \cdot \left( \eta_{0}(\beta^{N}) \cdots \eta_{n+m-1}(\beta^{N}) \right)^{\nu} > \frac{1}{\exp \left( \frac{\nu m_{g} \sigma \tau^{3}}{\tau^{1/ \kappa} - 1} \right)} \cdot s(n+m) \cdot \left( \eta_{0}(\beta^{1}) \cdots \eta_{n+m-1}(\beta^{1}) \right)^{\nu}
\end{equation*}
and
\begin{equation*}
    \Phi^{1}(\beta^{N}) > \frac{1}{\exp \left( \frac{\nu m_{g} \sigma \tau^{3}}{\tau^{1/ \kappa} - 1} \right)} \cdot \dfrac{ u( \eta_{n+m}(\beta^{N}) ) }{ u( \eta_{n+m}(\beta^{1}) ) } \Phi^{1}(\beta^{1}).
\end{equation*}
Since $\lim_{x \to s_{0}} u(x) = \infty$, the latter expression tends to $\infty$ as $N \to \infty$.
Now, write
\begin{align*}
    \Phi(\beta^{N}) &= \sum_{i=1}^{n+m-1} s(i) \cdot \left( \eta_{0}(\beta^{N}) \cdots \eta_{i-1}(\beta^{N}) \right)^{\nu} \cdot u(\eta_{i}(\beta^{N})) + \Phi^{1}(\beta^{N}) \\
    &\text{\hspace{14px}}+ \sum_{i=n+m+1}^{\infty} s(i) \cdot \left( \eta_{0}(\beta^{N}) \cdots \eta_{i-1}(\beta^{N}) \right)^{\nu} \cdot u(\eta_{i}(\beta^{N})).
\end{align*}
For $i \in \{ 1, 2, \ldots, n+m-1 \}$ and any $N>0$ we have $\eta_{i}(\beta^{N}) \in J_{1}$, thus $\eta_{i}(\beta^{N}) > \ell_{1} > s_{0}$. So, using Lemma~\ref{lem:u_bounded} we can bound the first term above independently of $N$:
\begin{align*}
    \left| \sum_{i=1}^{n+m-1} s(i) \cdot \left( \eta_{0}(\beta^{N}) \cdots \eta_{i-1}(\beta^{N}) \right)^{\nu} \cdot u(\eta_{i}(\beta^{N})) \right|
    < \sum_{i=1}^{n+m-1} \left| \underbrace{\left( 1 \cdots 1 \right)^{\nu}}_{\text{$i-1$ times}} \cdot \sup_{x \in [\ell_{1}, 1)}u(x) \right| < \infty.
\end{align*}
The third term converges, since the sum converges absolutely:
\begin{equation*}
    \sum_{i=n+m+1}^{\infty} \left| s(i) \cdot \left( \eta_{0}(\beta^{N}) \cdots \eta_{i-1}(\beta^{N}) \right)^{\nu} \cdot u(\eta_{i}(\beta^{N})) \right| = u(\beta^{1}) \cdot \sum_{i=n+1}^{\infty} \left( (\beta^{1})^{\nu} \right)^{i} < \infty
\end{equation*}
because $(\beta^{1})^{\nu} < s_{1} \leq 1$ for $\nu > 0$. This bound on the third term does not depend on $N$.
Thus the first and third term are bounded independently of $N$ while the second term is $\Phi^{1}(\beta^{N}) \to \infty$ as $N \to \infty$, so $\Phi(\beta^{N}) \to \infty$ as needed.
\end{proof}

We are now ready to prove Lemma~\ref{lem:5.18}.

\begin{proof}[Proof of Lemma~\ref{lem:5.18}]
Choose $m$ as provided by Lemma~\ref{lem:Phi_1_diff}. Increase $N$ by one at a time starting with $N=1$. We know that $\Phi(\beta^{1}) = \Phi(\omega) < \Phi(\omega) + \varepsilon$ and, by Lemma~\ref{lem:Phi_limit_unbounded}, there exists an $M$ with $\Phi(\beta^{M}) > \Phi(\omega) + \varepsilon$. Let $N$ be the smallest such $M$. Then $\Phi(\beta^{N-1}) \leq \Phi(\omega) + \varepsilon$, and by Lemma~\ref{lem:Phi_1_diff},
\begin{equation*}
    \Phi(\beta^{N}) < \Phi(\beta^{N-1}) + \varepsilon \leq \Phi(\omega) + 2\varepsilon.
\end{equation*}
Hence
\begin{equation*}
    \Phi(\omega) + \varepsilon < \Phi(\beta^{N}) < \Phi(\omega) + 2\varepsilon.
\end{equation*}
Choosing $\beta = \beta^{N}$ completes the proof.
\end{proof}

\bibliographystyle{amsalpha}
\bibliography{ref}
\end{document}